\documentclass[oneside]{amsart}
\usepackage[utf8]{inputenc}
\usepackage{booktabs}
\usepackage{color}
\newcommand{\href}[1]{#1} 
\usepackage{hyperref}
\usepackage{ifthen}
\usepackage{enumerate}
\usepackage{geometry}
\usepackage[all]{xy}
\usepackage{tikz}
\usepackage{tikz-cd}
\usepackage{longtable}
\usepackage{pdflscape}
\xyoption{arc}
\usepackage{amsmath,amssymb,amstext,latexsym,amsthm,amscd,mathrsfs, mathtools} 
\newtheorem{thm}{Theorem}[section]
\newtheorem*{thm*}{Theorem}
\newtheorem{lemma}[thm]{Lemma}
\newtheorem*{lemma*}{Lemma}
\newtheorem{prop}[thm]{Proposition}
\newtheorem*{prop*}{Proposition}

\theoremstyle{definition}
\newtheorem{df}[thm]{Definition}

\newtheorem{ex}[thm]{Example}

\newtheorem{rmk}[thm]{Remark}
\newtheorem{obs}[thm]{Observation}

\theoremstyle{remark}

\newcommand{\bbQ}{\mathbb{Q}}

\newcommand{\bbC}{\mathbb{C}}

\newcommand{\bbP}{\mathbb{P}}
\newcommand{\bbA}{\mathbb{A}}
\newcommand{\bbR}{\mathbb{R}}

\newcommand{\bbI}{\mathbb{I}}
\newcommand{\Z}{\mathbb{Z}}

\newcommand{\QQ}{\bbQ}
\newcommand{\Q}{\bbQ}

\newcommand{\C}{\bbC}

\renewcommand{\P}{\bbP}

\newcommand{\M}{\mathrm{M}}

\newcommand{\cO}{\mathcal{O}}

\newcommand{\cH}{\mathcal{H}}

\newcommand{\cM}{\mathcal{M}}

\newcommand{\sD}{{\mathscr{D}}}

\newcommand{\fp}{\mathfrak p}

\DeclareMathOperator{\Gal}{Gal}

\DeclareMathOperator{\GL}{GL}

\DeclareMathOperator{\SL}{SL}
\DeclareMathOperator{\PSL}{PSL}
\DeclareMathOperator{\SO}{SO}
\DeclareMathOperator{\Ort}{O}

\DeclareMathOperator{\sn}{sn}

\DeclareMathOperator{\No}{N}

\DeclareMathOperator{\Div}{Div}

\newcommand{\id}{\ensuremath \text{Id}}

\newcommand{\mtx}[4]{\left(\begin{matrix}#1&#2\\#3&#4\end{matrix}\right)}

\newcommand{\smtx}[4]{\left(\begin{smallmatrix}#1&#2\\#3&#4\end{smallmatrix}\right)}

\newcommand{\emphh}[2][ ]{%
\ifthenelse{\equal{#1}{ }}{\index{default}{#2}{\emph{#2}}}{\index{default}{#1@#2}{\emph{#2}}}%
}

\newcommand{\emphhh}[3][ ]{%
\ifthenelse{\equal{#1}{ }}{\index{default}{#2}{\emph{#3}}}{\index{default}{#1@#2}{\emph{#3}}}%
}

\def\sumprime{\mathop{\sum{\raise3pt\hbox{${}'$}}}} 
\def\prodprime{\mathop{\prod{\raise3pt\hbox{${}'$}}}}

\renewcommand{\epsilon}{\varepsilon}

\newcommand{\tto}[1]{%
\ifthenelse{\equal{#1}{}}{\to}{\stackrel{#1}{\longrightarrow}}}

\title[Evaluation of Bianchi rigid meromorphic cocycles at big ATR points]{Evaluation of Bianchi rigid meromorphic cocycles\\ at big ATR points}
\author{Lennart Gehrmann}
\address{Universität Bielefeld, Germany}
\email{gehrmann.math@gmail.com}
\author{Xavier Guitart}
\address{Universitat de Barcelona and CRM, Catalonia}
\email{xevi.guitart@gmail.com}
\author{Marc Masdeu}
\address{Universitat Autònoma de Barcelona and CRM, Catalonia}
\email{marc.masdeu@uab.cat}

\begin{document}
\begin{abstract}
  We develop the tools required to effectively evaluate the Bianchi rigid meromorphic cocycles introduced by Darmon--Gehrmann--Lipnowski at big ATR points, and use them to obtain the first numerical verification of the conjectured algebraicity of these special values.
Moreover, our computations suggest that these special values exhibit behaviour analogous to that of the special values of Borcherds products on Hilbert modular surfaces at big CM points.
	
\end{abstract}
\maketitle

\renewcommand{\emphh}[1]{{\emph{#1}{}}}

\newenvironment{mytable}[2]
{\begin{tabular}{|*{#1}{l|}}
\toprule
\multicolumn{#1}{c}{#2}\\
\midrule}
{\bottomrule\end{tabular}}

\section{Introduction} 

Singular moduli are the values of the classical $j$-function at CM points on the modular curve $X_0(1)$. These complex numbers are in fact algebraic and generate abelian extensions of the corresponding imaginary quadratic fields, and their norms down to $\Q$ exhibit remarkable and highly structured factorizations. Several of these properties of the $j$-function are known to hold in a more general setting, namely for the values of Borcherds products on Shimura varieties attached to orthogonal groups of real signature $(r,2)$ with $r\geq 1$ at special points.
Most notably, special values of Borcherds products on Hilbert modular surfaces at big CM points were studied by Bruinier and Yang in \cite{BY}. 

The theory of \emph{rigid meromorphic cocycles for orthogonal groups} was introduced in \cite{DGL} as a conjectural analogue for the case of groups $G$ of arbitrary real signature $(r,s)$, even in cases where no associated Shimura variety exists. This is a $p$-adic theory, rather than complex meromorphic, in the sense that rigid meromorphic cocycles are elements in  cohomology groups of the form $H^s(\Gamma,\cM_{\mathrm{rq}}^\times),$ where $\Gamma$ is a $p$-arithmetic group and  $\cM_{\mathrm{rq}}^\times$ is the multiplicative group of rigid meromorphic functions on a certain $p$-adic symmetric space $X_p$ attached to $G$. These cocycles can be evaluated at suitable special points on $X_p$, and the resulting $p$-adic values are conjectured to be algebraic, lying in abelian extensions of the reflex fields of the special points.

An important piece of evidence for the validity of the conjectures formulated in \cite{DGL} on the algebraicity of the values of rigid meromorphic cocycles at special points comes from numerical calculations for groups of signatures $(2,1)$ and $(3,1)$. In case of signature $(2,1)$, the orthogonal group is just an inner form of $\SL_2$ and $X_p=\cH_p$ is the $p$-adic upper half plane.
Thus, one recovers the rigid meromorphic cocycles introduced by Darmon--Vonk in \cite{darmon-vonk}, which have been extensively studied in recent articles such as \cite{DV25}, \cite{DPV}, \cite{DV22}, \cite{GD22}, \cite{DV3}, \cite{Geh}, and \cite{gmx}.

The first genuinely new instance of the construction arises in signature $(3,1)$, where $\Gamma$ is a Bianchi congruence subgroup of $\SL_2(\cO_K[1/p])$ with $K$ an imaginary quadratic field, in which $p$ is split, and $X_p=\cH_p\times \cH_p$ is the product of two copies of the $p$-adic upper half plane. In this setting, experimental evidence supporting the algebraicity conjecture was obtained in \cite{DGL} when $K=\Q(i)$ for two types of special points: the so-called \emph{small RM points}, which can be in fact regarded as ``images under the diagonal embedding'' of the special points considered in \cite{darmon-vonk}, and the so-called \emph{small CM points}, for which a few examples were calculated and found to be very close $p$-adically to algebraic numbers over the predicted fields of definition, with norms exhibiting remarkably small and concentrated support.

However, there is a third type of special points at which one can evaluate these Bianchi meromorphic cocycles, called \emph{big special points} in \cite{DGL}. The reflex field  of such points is an almost totally real (ATR) field, and their associated values were not computed numerically in \cite{DGL}. Therefore, the rationality conjecture remained untested for this type of points, which we will refer to as \emph{big ATR points} in this article.

The main goal of this article is to explain how the big ATR points can be computed in the case $K=\Q(i)$, and to present extensive numerical verifications of the algebraicity conjecture formulated in \cite{DGL} for these points. For completeness, we also extend the available numerical evidence in the cases of small RM and small CM points.

A bit more precisely, in \S\ref{sec: field of def} we set the stage by explaining how quadratic number fields give rise to four-dimensional quadratic spaces. Moreover, we determine the fields of definition of the nascent rigid meromorphic cocycles predicted by the conjectures of \cite{DGL}.
In \S\ref{sec: modular symbol} we recall the construction of a family of rigid meromorphic cocycles introduced in \cite[\S5.2]{DGL}, and we extend it in a slight but crucial way that enables evaluation at big ATR points. In \S\ref{sec: evaluation} we recall the construction of small RM points, small CM points, and big ATR points and describe how to perform this evaluation
Finally, in \S\ref{sec: experiments} we present the numerical results and discuss the supporting evidence for the algebraicity conjecture.
Furthermore, we exhibit a divisibility criterion for the primes in the support of these values that resembles the criterion for the special values of Borcherds products on Hilbert modular surfaces proven in \cite{BY}.

\subsubsection*{Acknowledgements} We thank Henri Darmon, Paul Kiefer and Michael Lipnowski for many valuable conversations and insights throughout the development of this project.
Xavier Guitart acknowledges the Spanish State Research Agency (PID2022-137605NB-I00).
Lennart Gehrmann acknowledges support by the Maria Zambrano grant for the attraction of international talent and by Deutsche Forschungsgemeinschaft (DFG, German Research Foundation) via the grant SFB-TRR 358/1 2023 -- 491392403.
Marc Masdeu acknowledges the Spanish State Research Agency (PID2020-116542GB-I00).
Guitart and Masdeu were supported by the Spanish State Research Agency, through the Severo Ochoa and María de Maeztu Program for Centers and Units of Excellence in R\&D (CEX2020-001084-M).

\section{Spinor norms and the field of definition}\label{sec: field of def}
The algebraicity conjecture for special values of rigid meromorphic cocycles of \cite{DGL} involves the compositum of two number fields: one is the field of definition of the point one is evaluating the rigid meromorphic cocycle at and the other is the field of definition of the rigid meromorphic cocycle itself.
The latter is a polyquadratic extension, whose definition involves certain local spinor groups.
In the following we will compute the local spinor groups in question and thus determine the field of definition for a large class of rigid meromorphic cocycles of four-dimensional quadratic spaces.

\subsection{The spinor norm}\label{sec: spin}
Let $F$ be a field of characteristic $\mathrm{char}(F)\neq 2$ and $(V,q)$ a finite-dimensional non-degenerate quadratic space over $F$.
Write $\langle \cdot, \cdot \rangle$ for the associated bilinear form, that is, $\langle v,v \rangle = 2 q(v)$ for all $v\in V$.
For any $a\in V$ with $q(a)\neq 0$ write $\tau_a \in \Ort(V)$ for the reflection 
\[
\tau_a\colon V\longrightarrow V,\ v \longmapsto v - 2 \frac{\langle a,v \rangle}{\langle a,a \rangle} a.
\]
Recall that the \emph{spinor norm} is the unique homomorphism
\[
\sn_F\colon \Ort(V)\longrightarrow F^\times / (F^\times)^2
\]
such that for any $a\in V$ with $q(a)\neq 0$ the reflection $\tau_a$ is mapped to $q(a) \bmod (F^\times)^2$.
Let $W\subseteq V$ be a non-degenerate subspace.
One can extend any $g\in \Ort(W)$ to an element in $\Ort(V)$ by letting it act as the identity on the orthogonal complement of $W$.
Under this embedding, the spinor norm on $\Ort(V)$ restricts to the spinor norm on $\Ort(W)$.
Given a non-zero isotropic vector $w\in V$, that is $q(w)=0$, and an element $a\in V$ with $\langle a,w\rangle=0$, one defines the \emph{Eichler transformation} via
\[
E(w,a)(v)\coloneqq v - \langle v, w\rangle a + \langle v, a\rangle w - q(a)\langle v, w \rangle w \quad \forall v\in V.
\]
Straightforward calculations show that $E(w,a)\in \SO(V)$, $E(w,w)=\id_V$ and that $E(w,a+b)=E(w,a)\circ E(w,b)$ for all $a,b\in V$ orthogonal to $w$.
As $E(w,a)=E(w,a/2)^2$, it follows that
\[
\sn_F(E(w,a))=1.
\]

Suppose that $R$ is a principal ideal domain with field of fractions equal to $F$.
An $R$-lattice $L$ in a quadratic $F$-vector space $(V,q)$ is called \emph{even} if $q(L)\subseteq R$.
It is called \emph{self-dual} if it agrees with its \emph{dual lattice}
\[
L^\sharp\coloneq \{v\in V\,\vert\, \langle v,L\rangle \subseteq R \}.
\]
For an even lattice we consider the subgroups $\Ort(L)\subseteq \Ort(V)$ and $\SO(L)\subseteq \SO(V)$ of elements that respect the lattice $L$.
For an isotropic vector $w\in L$ and $a\in L$ orthogonal to $w$, the Eichler transformation $E(w,a)$ belongs to $\SO(L)$.
The \emph{discriminant module} of an even lattice is the quotient
\[
D_L\coloneq L^\sharp/L,
\]
which is a torsion $R$-module.
The quadratic form $q$ descends to a quadratic map
\[
D_L \longrightarrow F/R,\quad x \longmapsto q(x) \bmod R.
\]
By a slight abuse of notation, we denote the quadratic map and the associated non-degenerate bilinear map by $q$ and $\langle\cdot,\cdot\rangle$, respectively.
The group of $R$-linear transformations of $D_L$ preserving $q$ is denoted $\Ort(D_L).$
The natural action of $\Ort(L)$ on $L^\sharp$ yields a homomorphism
\begin{equation}\label{reduction}
\Ort(L)\longrightarrow \Ort(D_L),
\end{equation}
which is surjective in case $R=\Z_p$ by \cite[Theorem 1.9.5]{Nikulin}.
If $L$ is a self-dual even lattice, one may define the integral spinor norm
\[
\sn_{R}\colon \Ort(L)\longrightarrow R^\times/(R^\times)^2
\]
as in \cite[Appendix C]{Conrad}.
Since the diagram
\[
\begin{tikzcd}
\Ort(L)\arrow[d, phantom, sloped, "\subset"] \arrow[r, "\sn_{R}"]&  R^\times/(R^\times)^2 \arrow[d]\\
\Ort(V)\arrow[r, "\sn_{F}"]& F^\times/(F^\times)^2
\end{tikzcd}
\]
is commutative, it follows that
\begin{equation}\label{selfdual}
\sn_R(\Ort(L))\subseteq R^\times \bmod (F^\times)^2.
\end{equation}

We want to study certain four-dimensional quadratic spaces attached to quadratic \'etale extensions.
To that end, let $K$ be an \'etale $F$-algebra of dimension $2$, i.e., $K$ is either a quadratic field extension or isomorphic to $F \times F$.
Given $\alpha \in K$ write $\alpha'$ for its Galois conjugate and $\No_{K/F}(\alpha)=\alpha \alpha'\in F$ for its norm.
Consider the quadratic space $V'_K$ whose underlying vector space is $K$ with quadratic form given by $\No_{K/F}$.
If $K$ is a field, $V'_K$ is anisotropic.
On the other hand, $V'_K$ is a hyperbolic plane in case $K= F \times F$.
Letting $a\in K^\times$ act on $V'_K$ by multiplying with $a/a'$ defines a homomorphism from $K^\times$ to $\SO(V'_K)$.
In fact, multiplication with $a/a'$ is equal to the composition $\tau_a \circ \tau_1$ (see for example \cite[Lemma 4.3]{DGL}).
One easily checks that this assignment defines an isomorphism
\begin{equation}\label{torusident}
K^\times/F^\times \xlongrightarrow{\sim} \SO(V'_K).
\end{equation}
By definition, the spinor norm of multiplication with $a/a'$ is equal to $\No_{K/F}(a) \bmod (F^\times)^2$.
Moreover, $\tau_1(a)=-a'$ for every $a\in K$ and, in particular, the equalities $\det(\tau_1)=-1$ and $\sn_F(\tau_1)=1$ hold.
In case $K= F\times F$, we will identify $K^\times/F^\times$ with $F^\times$ via the map $(a,b)\mapsto ab^{-1}$.

Consider the quadratic space $V_K\coloneq K\oplus F \oplus F$ with quadratic form $q$ given by $q(\alpha,b,c)= \No_{K/F}(\alpha) - bc$.
In other words, $V_K$ is the orthogonal direct sum of $V'_K$ and a hyperbolic plane.
We identify $V_K$ with a subspace of $M_2(K)$ by sending $(\alpha,b,c)$ to the matrix
\[
[\alpha;b,c]\coloneq \begin{pmatrix} \alpha & -b \\ c & -\alpha' \end{pmatrix}.
\]
Under this identification the quadratic form $q$ is given by the negative of the determinant.
The group $G\coloneq\{g \in \GL_2(K)\ \vert\ \det(g)\in F^\times\}$ acts on $V$ by isometries via
\[
g[\alpha;b,c]= g \cdot [\alpha;b,c] \cdot (g')^{-1}
\]
yielding a surjective homomorphism $G \rightarrow \SO(V_K)$.
The kernel of this homomorphism is given by the subgroup of scalar diagonal matrices with entries in $F^\times$.
Thus, the determinant $\det\colon G \rightarrow F^\times$ descends to a homomorphism
\[
\SO(V_K)\longrightarrow F^\times/(F^\times)^2,
\]
which can be shown to agree with the spinor norm.
In particular, the subgroup of $\SO(V_K)$ of elements with trivial spinor norm is identified with the group $\PSL_2(K)$.

\subsection{The field of definition}
Let $(V,q)$ be a quadratic $\Q$-vector space of dimension $n\geq 3$ and real signature $(r,s)$.
Moreover, for any subgroup $G\subseteq \Ort(V\otimes \bbR)$ we write $G^+\subseteq G$ for the subgroup of elements of trivial real spinor norm.
Fix a prime $p\geq 3$ such that $V\otimes \Q_p$ admits a self-dual $\Z_p$-lattice.
Let $X_p$ be the $p$-adic symmetric space attached to $V$ in \cite[Section 1.3]{DGL} (see also Section \ref{sec: spaces} below).
This rigid analytic space carries a natural action by the group $\SO(V\otimes \Q_p)$.
Moreover, consider the group $\Div_{\mathrm{rq}}^\dagger(X_p)$ of locally finite, rational quadratic divisors on $X_p$ (see \cite[Section 2.3]{DGL}).
Finally, let $\bbA^p_f$ be the restricted product of the non-Archimedean fields $\Q_\ell$, $\ell\neq p$, $\widehat{\Z}^p\subseteq \bbA^p_f$ the maximal compact subring, and consider the space
\[
S^p_f(V)\coloneq \{\Phi\colon V\otimes\bbA^p_f\rightarrow \Z\, \vert\, \Phi\, \mbox{is locally constant with compact support} \}
\]
of integer-valued Schwartz--Bruhat functions on $V\otimes \bbA^p_f$.
Given $\Phi \in S^p_f(V)$ we let $U_\Phi$ be the stabilizer of $\Phi$ in $\SO(V)$ and $\Gamma_\Phi=U_\Phi\cap \SO(V)^+$ the associated $p$-arithmetic subgroup.
Then, Section 2.4 of \cite{DGL} attaches to every positive rational number $d\in\Q_{>0}$ the so-called \emph{Kudla--Millson} divisor
\[
\mathscr{D}_{d,\Phi}\in\mathrm{H}^s(\Gamma_\Phi,\Div_{\mathrm{rq}}^\dagger(X_p)).
\]
The construction of Kudla--Millson divisors in special cases will be recalled in Section \ref{sec: MS} below.
These Kudla--Millson divisors are $p$-adic analogues of Heegner divisors on connected components of orthogonal Shimura varieties.
Write $\bbI$ for the idele group, which is the restricted product over the multiplicative groups $\Q_\ell^\times$ for all primes $\ell$ and $\bbR^\times$.
Taking the product of the local spinor norms defines a homomorphism
\[
\sn_{\bbA^p_f}\colon \SO(V_{\bbA^p_f})\longrightarrow \bbI.
\]
By analogy with fields of definition of connected components of orthogonal Shimura varieties we make the following definition (cf.~\cite[Section 4.3.8]{DGL}):
\begin{df}
The field of definition $E_\Phi$ of $\Phi$ is the polyquadratic extension of $\Q$ whose Galois group is identified with
\[
C(\Phi)\coloneq \Q^\times\backslash\bbI/\bbI^2 \sn_{\bbA^p_f}(U_\Phi)\Q_p^\times.
\]
via the Artin reciprocity map.
\end{df}
\begin{rmk}\label{splitremark}
Note that, by construction, the prime $p$ is split in $E_\Phi$ for every $\Phi\in S^p_f(V)$.
\end{rmk}

Let $K$ be an \'etale $\Q$-algebra of degree $2$ with ring of integers $\cO_K$.
In case $K=\Q\times \Q$ is the split \'etale algebra, $\cO_K$ denotes the order $\Z \times \Z$.
The quadratic space $V_K$ has signature $(3,1)$ if $K$ is imaginary quadratic and signature $(2,2)$ otherwise.
The base change of $V_K$ to $\Q_p$ admits a self-dual lattice if and only if $p$ is not ramified in $K$, which we will assume in the following.

Fix an integer $N\geq 1$ coprime to $p$ and put $\Z_N\coloneq\prod_{\ell\mid N} \Z_\ell$.
Moreover, let $\chi\colon (\Z/N\Z)^\times \to \{\pm 1\}$ be a quadratic Dirichlet character,
which we also view as a character on $\Z_N^\times$ via pullback.
Denote by $K_\chi$ the extension of $\Q$ cut out by $\chi$, which is at most of degree $2$.
Consider the $\Z$-lattice $L_K\coloneq\cO_K\times \Z \times \Z\subseteq V_K$ and its closure $\hat{L}^p_K$ in $V_K\otimes \bbA_f^p$.
The Schwartz--Bruhat function $\Phi_\chi \in S^p_f(V_K)$ is defined by
\[
\Phi_\chi(\alpha,b,c)=\begin{cases}
 \chi(c_N) &\mbox{if}\, (\alpha,b,c) \in \hat{L}^p_K,\, c_N\in  \Z_N^\times,\\
0 &\mbox{else},
\end{cases}
\]
where $c_N$ denotes the projection from $\widehat{\Z}^p$ to $\Z_N$.

\begin{prop}\label{fieldofdefprop}
The field of definition $E_{\Phi_\chi}$ of $\Phi_\chi$ is equal to $\Q$ unless $\chi$ is non-trivial, $K=K_\chi$, and $p$ is split in $K$, in which case $E_{\Phi_\chi}=K_\chi$.
\end{prop}
We will prove Proposition \ref{fieldofdefprop} in Section \ref{sec:proof of prop}.
\subsection{Local spinor norms}
From now on $F$ denotes a non-Archimedean local field with $\mathrm{char}(F)\neq 2$ and $K$ is an \'etale $F$-algebra of degree $2$.
Let $L'_K\subseteq V'_K$ be the even lattice given by the ring of integers $\cO_K$ of $K$.
Analogously to the case of the rational numbers, $\cO_K$ denotes the order $\cO_F \times \cO_F$ in case $K=F\times F$ is split.
The lattice $L'_K$ is self-dual if and only if $K/F$ is unramified.
\begin{lemma}\label{spinor1}
\begin{enumerate}[(a)]
\item\label{spinor1a} We have:
\[
\sn_{F}(\Ort(L'_K))=\sn_F(\SO(L'_K))=\begin{cases}
\No_{K/F}(K^\times)\bmod (F^\times)^2 &\text{if $K$ is a field},\\
\cO_F^\times \bmod (F^\times)^2&\text{if $K = K \times K$.}
\end{cases}
\]
\item\label{spinor1b}
Assume that $F=\Q_p$ for some prime $p$.
Then the homomorphism
\[
\SO(L'_K)\longrightarrow D_{L'_K}
\] 
is surjective.
\end{enumerate}
\end{lemma}
\begin{proof}
The reflection $\tau_1\in \Ort(L'_K)$ has trivial spinor norm but non-trivial determinant and, hence, $\sn_{F}(\Ort(L'_K))=\sn_F(\SO(L'_K))$.
In case $K$ is a field, $a/a'$ belongs to $\cO_K^\times$ and, thus, defines an automorphism of the lattice $L'_K$.
If $K=F\times F$, we may assume that $a$ is of the form $(b,1)$, $b\in F^\times$.
Since $a/a'=(b,b^{-1})$, this defines an automorphism of the lattice $L'_{F\times F}$ if and only if $b\in \cO_F^\times$.

Part \eqref{spinor1b} is trivial in case $K/\Q_p$ is unramified.
So assume that $K/\Q_p$ is a ramified quadratic extension.
In that case $\cO_K=\Z_p+ \Z_p\sqrt{d}$ for some $d\in \Z_p$.
In particular, $\Z_p$ with the quadratic form $x \mapsto x^2$ splits off as an orthogonal direct summand of $\cO_K$.
This lattice is self-dual for $p\neq 2$, while for $p=2$ its discriminant module is one-dimensional over $\Z/2\Z$.
In any case, the reflection $\tau_1$ or in other words, multiplication by $-1$ on $\Z_p$, acts trivially on $D_{L'_K}$.
Thus, the claim follows from the surjectivity of \eqref{reduction}.
\end{proof}

Next we compute the group of spinor norms for the lattice $L_K\coloneq \cO_K \times \cO_F \times \cO_F \subseteq V_K$.
\begin{lemma}\label{spinor2}
We have:
\[
\sn_{F}(\Ort(L_K))=\sn_F(\SO(L_K))=\begin{cases}
F^\times \bmod (F^\times)^2 &\mbox{if}\,K\,\mbox{is ramified},\\
\cO_F^\times \bmod (F^\times)^2&\mbox{else}.
\end{cases}
\]
\end{lemma}
\begin{proof}
The first equality follows by considering the reflection attached to $1\in K$ as in the proof of Lemma \ref{spinor1}.
As the lattice $L_K$ splits of a hyperbolic plane, we have $\cO_F^\times \bmod (F^\times)^2 \subseteq \sn_F(\Ort(L_K))$ in any case by Lemma \ref{spinor1}.
If $K$ is ramified, Lemma \ref{spinor1} also implies that a uniformizer of $F$ lies in $\sn_F(\Ort(L_K))$.
So assume that $K$ is not ramified.
Then the lattice $L_K$ is self-dual and we may conclude by invoking \eqref{selfdual}.
\end{proof}

Let $\chi\colon \cO_F^\times\to \{\pm 1\}$ be a continuous character.
If $\chi$ is trivial, we put $K_\chi=F$.
Otherwise, $K_\chi$ denotes the unique quadratic extension of $F$ with $\No_{K_\chi/F}(K_\chi^\times)\cap \cO_F^\times =\ker(\chi)$, which is necessarily ramified. 
Similarly as above, consider the Schwartz--Bruhat function $\varphi_\chi\colon V_K\rightarrow \Z$ given by
\[
\varphi_\chi(\alpha,b,c)=\begin{cases}
 \chi(c) &\mbox{if}\, (\alpha,b,c) \in L_K,\, c\in \cO_F^\times,\\
0 &\mbox{else},
\end{cases}
\]
and its stabilizer $U_\chi\subseteq \SO(V_K)$.
Note that $U_\chi$ sends the set $\{v\in V_K\,\vert\, \varphi_{\chi}(v)=1\}$ to itself.
As this set generates $L_K$, we see that $U_\chi$ is a subgroup of $\SO(L_K)$.
\begin{lemma}\label{spinor4}
We have:
\[\sn_{\Q_p}(U_\chi)=
\begin{cases}
\cO_F^\times \bmod (F^\times)^2&\mbox{if}\, K\, \mbox{is unramified},\\
F^\times \bmod (F^\times)^2&\mbox{if}\, K\, \mbox{is ramified and}\, K\ncong K_\chi.\\
\end{cases}
\]
\end{lemma}
\begin{proof}
Assume first that $K$ is unramified.
Then by Lemma \ref{spinor1}, $\sn_{F}(U_\chi)$ contains $\cO_F^\times \bmod (F^\times)^2$, while it is contained in the latter by Lemma \ref{spinor2}.
In the ramified case, Lemma \ref{spinor1} again implies that $\sn_{F}(U_\chi)$ contains $\No_{K/F}(K^\times)\bmod (F^\times)^2$.
The action of $\ker(\chi)\subseteq \cO_F^\times$ on the hyperbolic plane spanned by the coordinates $b$ and $c$ clearly stabilizes $\Phi_{\chi}$.
In case $K\neq K_\chi$ one easily sees that $\ker(\chi)$ and $\No_{K/F}(K^\times)$ together generate $F^\times \bmod (F^\times)^2$.
\end{proof}

\subsection{Proof of Proposition \ref{fieldofdefprop}}\label{sec:proof of prop}
By \eqref{torusident}, the orthogonal group of the hyperbolic plane spanned by the coordinates $b$ and $c$ can be identified with the multiplicative group.
Under this identification $\ker(\chi)\subseteq (\widehat{\Z}^p)^\times$ clearly stabilizes $\Phi_\chi$.
In particular, $\ker(\chi)$ is contained in $\sn_{\bbA^p_f}(U_{\Phi_\chi})$, which implies that $E_{\Phi_\chi}\subseteq K_\chi$.
From Remark \ref{splitremark} and Lemma \ref{spinor4} it follows that $E_{\Phi_\chi}=\Q$ unless $p$ is split in $K_\chi$ and $K=K_\chi$.
So let us assume that this is the case.
Consider the even lattice $L_K(N)\coloneq \{(\alpha,b,c) \in L_K\, \vert\, N \mid C\}$.
Its discriminant module splits into the orthogonal direct sum $D_{L_K(N)}=D_{L'_K}\oplus H$, where $H$ is the free $\Z/N\Z$-module of rank $2$ with generators $e_1$ and $e_2$ having representatives $b/N$ and $c$,respectively, and therefore satisfying
\[
q(e_1)=q(e_2)=0\bmod \Z\quad \mbox{and}\quad \langle e_1, e_2\rangle =-1/N \bmod\Z.
\]
One easily checks that the group $U_{\Phi_\chi}$ stabilizes the completion of the lattice $L_K(N)$ in $V_K\otimes \bbA^p_f$ and, hence, it acts on $D_{L_K(N)}$.
Since it also stabilizes the completion of the lattice $L_K$, it stabilizes the $\Z/N\Z$-line generated by $e_2$.
By multiplying with an element of $\ker(\chi)$ (viewed as a subgroup of the automorphisms of the hyperbolic plane spanned by $b$ and $c$), it is enough to consider those $g\in U_{\Phi_\chi}$ that stabilize $e_2$.
For such $g$ we compute
\[
\langle g\alpha, e_2\rangle=\langle g\alpha, ge_2\rangle=\langle \alpha, e_2 \rangle=0\quad\forall \alpha \in D_{L'_K}.
\]
In particular, $g\alpha$ is of the form $\alpha_g+g_\alpha e_2$ for some $\alpha_g\in D_{L'_K}$ and $g_\alpha \in \Z/N\Z$.
As $q(\alpha)=q(g\alpha)=q(\alpha_g+g_\alpha e_2)=q(\alpha_g)$, the assignment $\alpha \mapsto \alpha_g$ defines an element of $\Ort(D_{L'_K})$.
Using Lemma \ref{spinor1}, we may reduce to the case that $\alpha_g=\alpha$ for all $\alpha \in D_{L'_K}$.
By a similar computation as above, we see that $ge_1=e_1+g_{e_1}e_2+\beta$ for some $g_{e_1}\in \Z/n\Z$ and $\beta\in D_{L'_K}$.
As $ge_1$ is isotropic and orthogonal to $g\alpha$, one deduces that
\[
q(\beta)=- g_\beta/N\quad\mbox{and}\quad\langle \alpha, \beta \rangle= -g_\alpha/N.
\]
Let $\widetilde{\beta} \in L_K(N)^\sharp$ be a lift of $\beta$.
Since $N\widetilde{\beta}\in L_K(N)$, the Eichler transformation $E(c,N\widetilde{\beta})$ belongs to $\SO(L_K(N))$.
One readily checks that $E(c,N\widetilde{\beta})$ and $g$ induce the same action on $D$.
Since Eichler transformations have trivial spinor norm, we may assume that $g$ acts trivially on $D$.
But every $g$ that acts trivially on $D$ satisfies $\sn_{\bbA_f^p}(g)\in \ker(\chi)$:
in case $p\neq 2$, this follows from \cite[Theorem 9.3]{MM}, while in case $p=2$, it follows from Theorem 10.7 and Theorem 10.8 of \cite{MM}.
Thus, the assertion follows.

\section{Rigid meromorphic cocycles for the Gaussian rationals}\label{sec: modular symbol}

Let $p\equiv 1\pmod 4$ be a prime, and let $K=\Q(i)$ be the field of Gaussian rationals. In this section we recall the construction and algorithmic computation of a family of rigid meromorphic cocycles  or, more precisely, rigid meromorphic modular symbols, for a congruence subgroup $\Gamma \subseteq \SL_2(\cO_K[1/p])$ as described in \cite[\S5.2]{DGL}, and we complement it in one essential aspect. In \textit{loc.cit.~}only the evaluation of these modular symbols at the pair of cusps $(0,\infty)$ was provided. While this suffices for evaluating at small CM and small RM points, evaluation at big ATR points requires the values of the modular symbol at other pairs of cusps. This is precisely the extension that we develop here.

Let $V=V_K$ be the four dimensional $\Q$-vector space
\begin{align*}
  V = \left\{[\alpha,b,c]=\smtx{\alpha}{-b}{c}{-\alpha'}\in \mathrm{M}_2(K)\ \middle\vert\ \alpha\in K, b,c\in\Q\right\},
\end{align*}
where $\alpha'$ denotes the $\Gal(K/\Q)$-conjugate of $\alpha$.
Regard $V$ as a quadratic space over $\Q$ with quadratic form $q$ given by $q([\alpha,b,c])=-\det([\alpha,b,c])=\No_{K/\Q}(\alpha)-bc$.
Recall from Section \ref{sec: spin} that the group $\SL_2(K)$ acts on $V$ by means of
\begin{align*}
  \gamma * M = \gamma M(\gamma')^{-1}
\end{align*}
and this action identifies $\PSL_2(K)$ with the subgroup of elements in $\SO(V)$ with trivial spinor norm.

\subsection{Symmetric spaces}\label{sec: spaces}
The symmetric space $X_\infty$ of the real Lie group $\SO(V\otimes \bbR)$ is given by the space of negative definite lines in $V\otimes \bbR$.
It carries a natural action of $\SO(V\otimes\bbR)$ and, therefore, it also carries an action of $\SL_2(\bbC)$.
For an element $M\in V$ with $q(M)>0$ we write $\Delta_{M,\infty}$ for the subspace of $X_\infty$ of all lines orthogonal to $M$.
We identify $X_\infty$ with hyperbolic $3$-space via the map
\[
X_\infty \xlongrightarrow{\sim} \bbC \times \bbR_{>0},\quad [\alpha,b,c] \longmapsto (\alpha/c, -q([\alpha,b,c])/c^2).
\]
Under this identification $\SL_2(\bbC)$ acts on $(z,t)\in \bbC \times \bbR_{>0}$ via
\[
\begin{pmatrix} a & b \\ c & d \end{pmatrix}.(z,t)=\left(\frac{(az+b)\overline{(cz+d)}+act^2)}{|cz+d|^2+|c|^2t^2},\frac{t}{|cz+d|^2+|c|^2t^2}\right).
\]
Let $\overline{X}_\infty$ be the space obtained by attaching $\mathbb{P}^1(K)$ to $X_\infty$ by mapping $\alpha\in K$ to $(\alpha,0)$ and $\infty$ to $(0,\infty)$.
The action of $\SL_2(K)$ on $X_\infty$ extends continuously to an action on $\overline{X}_\infty$.
On the boundary the action is given by fractional linear transformations.
Given $r,s\in \mathbb{P}^1(K)$, we denote by $(r,s)$ the open hyperbolic geodesic path from $r$ to $s$ in $X_\infty$.
Let $\Z[\mathbb{P}^1(K)]_0\subseteq \Z[\mathbb{P}^1(K)]$ the subgroup of divisors of degree $0$.
Given an abelian group $M$, an \emph{$M$-valued modular symbol} is a $\Z$-linear homomorphism from $\Z[\mathbb{P}^1(K)]_0$ to $M$.
In case $M$ carries a $\Gamma$-action for some subgroup $\Gamma\subseteq \GL_2(K)$, we endow the space $\mathcal{MS}(M)$ of $M$-valued modular symbols with the $\Gamma$-action given by $(\gamma F)(D)=\gamma(F(\gamma^{-1}D))$.
Write
\begin{equation}\label{delta}
\delta\colon \mathcal{MS}^{\Gamma}(M)\longrightarrow \mathrm{H}^1(\Gamma,M)
\end{equation}
for the first boundary map in the long exact sequence in group cohomology induced by the short exact sequence
\[
0 \longrightarrow \Z[\mathbb{P}^1(K)]_0 \longrightarrow \Z[\mathbb{P}^1(K)] \xlongrightarrow{\deg} \Z \longrightarrow 0.
\]
A modular symbol is uniquely characterized by its values on elements of the form $[s]-[r]$, which we identify with the geodesic $(r,s)$.
A function $F$ on these open hyperbolic geodesics comes from a modular symbol if and only if
\[
F(r,s)=-F(s,r)\quad\mbox\quad F(r,s)+F(s,t)+F(t,r)=0\ \forall r,s,t\in\mathbb{P}^1(K).
\]

Let $Q$ be the zero locus of $q$ inside the projective space of $V$.
In other words, $Q$ parametrizes isotropic lines in $V$.
The $p$-adic symmetric space $X_p$ attached to $V$ is the connected rigid analytic subspace of the base change of $Q$ to $\Q_p$, whose $\bbC_p$-valued points are given by
\[
X_p(\bbC_p)=\{\xi\in X_p\,\vert\, \langle\xi, w \rangle= 0\ \forall w\in Q(\Q_p) \}.
\]
See \cite[Section 1.3]{DGL} for more details on the construction.
As in the Archimedean case, for $M\in V$ with $q(M)>0$ we consider the divisor $\Delta_{M,p}\subseteq X_p$ given by all isotropic lines orthogonal to $M$. 
Now fix once and for all an embedding $\overline{\Q}\hookrightarrow \bbC_p$.
As the prime $p$ is split in $K$, this embedding induces the isomorphism
\[
K\otimes \Q_p \xlongrightarrow{\sim} \Q_p \times \Q_p,\quad \alpha\otimes x \longmapsto (\alpha x,\alpha'x).
\]
Utilizing this isomorphism we identify $X_p$ with the square of the $p$-adic upper half space $\cH_p$ via the assignment
\[
\cH_p\times\cH_p\xlongrightarrow{\sim} X_p,\quad (\tau_1,\tau_2)\longmapsto [(\tau_1,\tau_2),\tau_1\tau_2,1].
\]
Under this identification the divisor $\Delta_{M,p}$ of a matrix $M=[\alpha,b,c]\in V$ with $q(M)>0$ is given by the vanishing locus of the function
\[
F_M\colon \cH_p\times \cH_p \longrightarrow \bbC_p,\quad (\tau_1,\tau_2) \longmapsto c\tau_1\tau_2-\alpha\tau_1-\alpha'\tau_2+b.
\]

\subsection{Modular symbols}\label{sec: MS}
Denote by $L$ the $\Z$-lattice $V\cap \M_2(\cO_K)$ and similarly write $L[1/p]$ for the $\Z[1/p]$-lattice $V\cap \M_2(\cO_K[1/p])$.
Let $\chi_4$ be the odd Dirichlet character modulo 4, and for $M=[\alpha,b,c]\in L[1/p]$ put $\chi_4(M) = \chi_4(c)$.
Observe that $\chi_4(p)=1$, so the quantity $\chi_4(c)$ makes sense even if $c$ has denominators, simply by applying $\chi_4$ to the numerator of $c$.
Moreover, let $\Gamma\subseteq \SL_2(\cO_K[1/p])$ be the subgroup of matrices that are upper triangular modulo $2$.
Denote by $\Div^\dagger(X_p)$ the group of locally finite, rational quadratic divisors on $X_p$ as defined in \cite[Section 2.3.3]{DGL}.
For every integer $d>0$ that is not a norm of an element in $K$, it is proved in Section 2.6 of \textit{loc.cit.~}that the assignment
\begin{equation}\label{MSdiv}
(r,s)\longmapsto \sum_{\substack{ M\in L[1/p] \\ q(M)=d}} \chi_4(M)(\Delta_{M,\infty}\cap (r,s)) \cdot \Delta_{M,p}
\end{equation}
defines a $\Gamma$-invariant modular symbol with values in $\Div^\dagger(X_p)$.
Here, $\Delta_{M,\infty}\cap (r,s)$ is the intersection product between the cycle $\Delta_{M,\infty}$ and the hyperbolic geodesic path  $(r,s)$ joining $r$ and $s$ in $X_\infty$ as defined in \cite[Section 2.6]{DGL}.
The image of the modular symbol \eqref{MSdiv} under the boundary map \eqref{delta} is equal to the Kudla--Millson divisor $\mathscr{D}_{d,\Phi_{\chi_4}}$.

Let $\cM^\times$ be the group of multiplicative rigid meromorphic functions on $X_p=\cH_p\times \cH_p$.
For every positive integer $d$ that is not a norm from $K$, we consider the formal modular symbol $J_d$ for $\Gamma$ with values in $\cM^\times$ defined in \cite[\S5.2]{DGL}, which evaluated at a pair  $r,s\in \mathbb{P}^1(K)$ is given by
\begin{align}\label{eq: J d}
  J_d(r,s)(\tau_1,\tau_2) = \prod_{\substack{M\in L[1/p] \\ q(M)=d}}F_M(\tau_1,\tau_2)^{\chi_4(M)(\Delta_{M,\infty}\cap (r,s))}.
\end{align}
Its image under the divisor map is clearly equal to \eqref{MSdiv}.
For $\sD=\sum n_i(d_i)$ a formal sum of positive integers each not a norm from $K$, define
\begin{align}\label{eq: Phi D}
  J_\sD (r,s)= \prod_i J_{d_i}(r,s)^{n_i}.
\end{align}
As explained in \cite[\S5.2]{DGL}, if the assignment $(r,s)\mapsto J_{\sD}(r,s)$ is $\SL_2(\cO_K)$-equivariant, then the theory of modular symbols reduces the calculation of $J_\sD$ to that of $J_d(0,\infty)$. However, due to the presence of the character $\chi_4$ in the exponent of \eqref{eq: J d} one can not expect equivariance by the full $\SL_2(\cO_K)$, but only by the congruence subgroup
\[
  \Gamma_0(2)=\left\{\mtx x y z t \in\SL_2(\cO_K)\colon 2\mid z\right\}.
  \]
Therefore, the theory of modular symbols reduces the calculation of $J_\sD$ to that of $J_\sD(g_j0,g_j\infty)$ for some system of representatives $g_1,\dots,g_h$ of $\Gamma_0(2)\backslash \SL_2(\cO_K)$.

In \cite[\S5.2]{DGL}, conditions on $\sD$ are given in order that the formal product \eqref{eq: Phi D} converges when evaluated at the pair of cusps $(0,\infty)$. This is enough for computing values at small RM or small CM points. Indeed, for these points one needs to evaluate $J_\sD(r,s)$ where the pair $(r,s)$ has the special feature that the path $(r,s)$ can be written as
\begin{align*}
  (r,s) = \sum_j ( \gamma_j0, \gamma_j \infty) = \sum_j \gamma_j( 0,  \infty)
\end{align*}
for some matrices  $\gamma_j$ that belong to $\Gamma_0(2)$. But this is not true in general for an arbitrary pair of cusps $(r,s)$ and, in particular, for the cusps that appear when evaluating at big ATR points.

In order to evaluate the cocycle at an arbitrary pair of cusps $(r,s)$, one can proceed as follows. First of all, use the Manin trick to write 
\begin{align*}
  (r,s) = \sum_j ( \gamma_j 0, \gamma_j\infty )
\end{align*} 
for some $\gamma_j\in\SL_2(\cO_K)$. Let $\{g_1,\dots, g_h\}$ be a set of representatives for $\Gamma_0(2)\backslash \SL_2(\cO_K)$ and write $\gamma_j=\gamma_{j}' g_{k(j)}$, with $\gamma_j'\in \Gamma_0(2)$ and some $k(j)\in \{1,\dots,h\}$.
By $\Gamma_0(2)$-equivariance we have that
\begin{align*}
  J_\sD(\gamma_j' g_{k(j)} 0,\gamma_j' g_{k(j)}\infty)= (\gamma_j')^{-1}. J_\sD(g_{k(j)} 0, g_{k(j)}\infty).
\end{align*}
Therefore, knowing the values $J_\sD(g_j0,g_j\infty)$ for $j=1,\dots,h$ suffices to compute $J_\sD(r,s)$. In fact, even less information is required, a fact that is crucial for accelerating the computations. Indeed, since $[\SL_2(\cO_K)\colon \Gamma_0(2)]=6$, one might expect that the values at six pair of cusps are required; however, as we show in the next result, it is enough to know $J_\sD$  at only two pairs of elements in $\mathbb{P}^1(K)$ instead of six. For this, it will be important to consider also the opposite path to $(r, s)$, that we denote by $-(r,s)$.
\begin{lemma}
  Let $g = \smtx{i}{0}{i+1}{-i}\in \SL_2(\cO_K)$ and let $r,s\in\P^1(K)$. The path $(r, s)$ can be written as
  \begin{small}
  \begin{align}\label{r s as a sum}
(r, s)=    \sum_j \pm \gamma_j(0, \infty ) + \sum_k \pm \gamma_k (g0, g\infty)= \sum_j \pm \gamma_j(0, \infty ) + \sum_k \pm \gamma_k (0,\frac{1+i}{2}),
  \end{align}
\end{small}
for some matrices $\gamma_j,\gamma_k\in \Gamma_0(2)$; we remark that the $\pm$ sign in the above expression refers to the fact that each of the paths in the sum can be either as it appears or its opposite.
\end{lemma}
\begin{proof}A direct computation shows that a system of representatives for $\Gamma_0(2)\backslash \SL_2(\cO_K)$ is given by:
  \begin{align*}
  g_1&= \smtx{1}{0}{0}{1},\, g_2=\smtx{0}{i}{i}{0},\, g_3=\smtx{i}{0}{i + 1}{-i},\\  g_4&=\smtx{i}{-i}{-2 i + 1}{ i - 1},\, g_5=\smtx{i}{0}{i + 2}{-i},\, g_6= \smtx{i}{0}{2 i - 1}{ -i}.
  \end{align*}
  By the Manin trick, it is enough to find an expression of the form \eqref{r s as a sum} for the cusps of the form $(g_j0,g_j\infty)$. This is obviously true for $j=1,3$. For the rest, one can directly check that
  \begin{align*}
    (g_20,g_2\infty)&=-\smtx{i}{0}{0}{-i}(0,\infty)\\
    (g_40,g_4\infty)&=-\smtx{i-2}{1}{4}{-i-2}(0,\frac{1+i}{2})\\
    (g_50,g_5\infty)&=-\smtx{-i}{i}{-2}{i+2}(0,\infty)+ \smtx{i}{0}{2}{-i}(0,\infty)\\
    (g_60,g_6\infty)&=-\smtx{1}{i}{2i+2}{2i-1}(0,\infty)+\smtx{i}{0}{2i-2}{-i}(0,\infty).
  \end{align*}
\end{proof}

Following \cite[\S5.2]{DGL}, it is convenient for the algorithmic calculation to decompose
\[ J_d(r,s) = \prod_{j=0}^\infty\Phi_{d,j}(r,s),
\]
  where the $\Phi_{d,j}(r,s)$ are given by
\begin{align}\label{eq: Phi d j}
\Phi_{d,j}(r,s)=  \prod_{\substack{M\in L \\ q(M) =dp^{2j}}}F_M(\tau_1,\tau_2)^{\chi_4(M)\Delta_{M,\infty}\cap (r,s)}.
\end{align}
The set
\begin{align*}
   \Sigma_{d,j}(0,\infty) = \{ M\in L\,\vert\, q(M) =dp^{2j} \ \text{and} \ \Delta_{M,\infty}\cap (0,\infty)\neq 0 \}
\end{align*}
can be computed algorithmically. Indeed, by \cite[Lemma 5.1]{DGL} we have that
\begin{align}\label{eq: sigma d j zero infty}
\Sigma_{d,j}(0,\infty) = \{M=\smtx{\alpha}{-b}{c}{-\alpha}\in M_2(\cO_K) \colon \alpha\alpha'-bc=p^{2j}, \, bc<0\},
\end{align}
which is obviously finite and effectively computable since there are only finitely many possibilities for $b$ and $c$, and for every choice of $b$ and $c$ there are finitely many possibilities for $\alpha$.
For any pair of elements $r,s\in \mathbb{P}^1(K)$, the set
\begin{align*}
  \Sigma_{d,j}(r,s) = \{ M\in  M_2(\cO_K)\, \vert\, q(M) =dp^{2j}  \ \text{and} \ \Delta_{M,\infty}\cap (r,s)\neq 0 \}
\end{align*}
can also be computed, using the following two properties of the intersection product:
\begin{align*}
&\Delta_{M,\infty}\cap (r,s) = \Delta_{\gamma M,\infty}\cap (\gamma r,\gamma s) \text{ for all }\gamma\in \SL_2(\cO_K)\ \mbox{and} \\
 & \Delta_{M,\infty}\cap (r,t) = \Delta_{M,\infty}\cap (r,s) +\Delta_{M,\infty}\cap (s,t),\ \mbox{for all}\ r,s,t\in\P_1(K).
\end{align*}
Indeed, using the Manin trick, we write
  \begin{align*}
    (r,s) = \sum_{j=1}^n \gamma_j(0,\infty) \text{ with } \gamma_j\in\SL_2(\cO_K).
  \end{align*}
  Using the two properties above, we see that
  \begin{align}\label{eq}
    \Delta_{M,\infty}\cap (r,s) =\sum \Delta_{M,\infty}\cap (\gamma_j 0,\gamma_j \infty)=\sum \Delta_{\gamma_j^{-1}M,\infty}\cap ( 0, \infty).
  \end{align}
  From this, we deduce that $\gamma_j^{-1}M = M_0$ for some $M_0\in\Sigma_{d,j}(0,\infty)$ if $\Delta_{M,\infty}\cap (r,s)\neq 0$. Therefore, we can compute $\Sigma_{d,j}(r,s)$ running over matrices $M\in \cup_{i=1}^n\gamma_j\Sigma_{d,j}(0,\infty)$ (there are finitely many of these and we can compute them since we can compute $\Sigma_{d,j}(0,\infty)$), and for each $M$ we compute $\Delta_{M,\infty}\cap (r,s)$ using \eqref{eq}; if $\Delta_{M,\infty}\cap (r,s)\neq 0$ we put $M$ in $\Sigma_{d,j}(r,s)$ and if $\Delta_{M,\infty}\cap (r,s)=0$ we discard it.

By the Manin trick and $\Gamma_0(2)$-equivariance, it is therefore enough to compute $J_\sD(0,\infty)$ and $J_\sD(g 0,g \infty)$ for $g = \smtx{i}{0}{i+1}{-i}$ in order to compute the cocycle $J_\sD$.

\subsection{Convergence criterion}
In \cite[\S5.2]{DGL}, conditions on the divisors $\sD$ are found to ensure that $J_\sD(0,\infty)$ converges. We next show that one extra condition is needed to guarantee that $J_\sD(g0,g\infty)$ also converges. Define the weight $\operatorname{wt}(\Phi_{d,j}(r,s))$ of $\Phi_{d,j}(r,s)$ to be the sum of the exponents that appear in \eqref{eq: Phi d j}.
The following lemma compares the weights of $\Phi_{d,j}(0,\infty)$ and $\Phi_{d,j}( g  0, g  \infty)$.

\begin{lemma}\label{lemma:degree} We have:
\[
\operatorname{wt}(\Phi_{d,j}(0,\infty))=
(-1)^d\operatorname{wt}(\Phi_{d,j}( g  0, g  \infty)).
\]
\end{lemma}
\begin{proof}On the one hand, as in the proof of \cite[Lemma 5.2]{DGL} we have that
  \begin{align*}
    \operatorname{wt}(\Phi_{d,j}( 0, \infty))    =\sum_{\substack{\alpha\alpha'-bc=dp^{2j} \\ bc<0}}\chi_4(c)\operatorname{sign}(c).
  \end{align*}
  On the other hand,
  \begin{align*}
    \operatorname{wt}(\Phi_{d,j}( g  0, g  \infty)) 
		&= \sum_{\substack{M\in \Sigma_{d,j}(r,s)\\ q(M) =dp^{2j}}}\chi_4(M)(\Delta_{M,\infty}\cap ( g  0, g  \infty))\\ 
		&=\sum_{\substack{M\in  \Sigma_{d,j}(0,\infty) \\ q(M) =dp^{2j}}}\chi_4( g  M)(\Delta_{ g  M,\infty}\cap ( g  0, g  \infty))\\ 
		&=\sum_{\substack{M\in  \Sigma_{d,j}(0,\infty) \\ q =dp^{2j}}}\chi_4( g  M)(\Delta_{ M,\infty}\cap ( 0, \infty))\\
     &=\sum_{\substack{\alpha\alpha'-bc=dp^{2j} \\ bc<0}}\chi_4(c')\operatorname{sign}(c),
  \end{align*}
 where for a matrix  $M=\smtx{\alpha}{-b}{c}{-\alpha'}= \smtx{r+is}{-b}{c}{-r+is}$ we denote by $c'$  the lower left entry of $ g  M$. A direct computation shows that
 \begin{align}\label{eq: c prime}
    c' = -2(r+s)+2b+c.
  \end{align}
  We next show that $\chi_4(c')=\chi_4(c)$ if $d$ is even and that $\chi_4(c')=-\chi_4(c)$ if $d$ is odd, from which the result follows immediately. To begin with, clearly $c$ is even if and only if $c'$ is even, so that $\chi_4(c)=0$ if and only if $\chi_4(c')=0$.  So we can assume from now on that $c$ is odd. 

  Suppose now that $d$ is even. If $b$ is even, the fact that
  \begin{align}\label{eq: r2 + s2}
    r^2+s^2-b c= d p^{2j}
  \end{align}
  implies that $r$ and $s$ have the same parity, and then \eqref{eq: c prime} shows that $c'\equiv c\pmod 4$. If $b$ is odd, then $r$ and $s$ have different parity, and \eqref{eq: c prime} implies again that $c'\equiv c\pmod 4$. 

  Finally, suppose that $d$ is odd. If $b$ is even, \eqref{eq: r2 + s2} forces $r$ and $s$ to have different parity and  \eqref{eq: c prime} then gives $c'\equiv -c\pmod 4$. If $b$ is odd, $r$ and $s$ have the same parity and then \eqref{eq: c prime} also gives $c'\equiv -c\pmod 4$
\end{proof}

Let $\sD = \sum_i n_i (d_i)$ be a divisor and put $\Phi_{\sD,j}=\prod_i \Phi_{d_i,j}^{n_i}$. By \cite[Lemma 5.2]{DGL}, we have that
\[
  \mathrm{wt}(\Phi_{d,j}(0,\infty))= 16\sigma(dp^{2j}),\text{ where }\sigma(n)=\sum_{4\nmid e\mid n}e.
\]
Combining this with Lemma \ref{lemma:degree}, we see that if
\begin{align}\label{eq:cond1}
 \sum_i n_i \sigma(d_i)= 0 \text{ and }  \sum_i (-1)^{d_i}n_i \sigma(d_i)= 0 ,
\end{align}
then both $\Phi_{\sD,j}(0,\infty)$ and $\Phi_{\sD,j}(g0,g\infty)$  have weight 0 for all $j$. Exactly as in \cite{DGL}, however, this weight-zero condition is not sufficient to guarantee convergence of the infinite products $\prod_{j}\Phi_{\sD,j}(0,\infty) $ and $\prod_{j}\Phi_{\sD,j}(g0,g\infty) $. The Borcherds theory developed in \cite{DGL} suggests that the additional obstruction comes from cusp forms of weight 2 and level $4p$.

Indeed, if $f=\sum_{n}a_n(f)q^n$ denotes the $q$-expansion of such a form, one must also impose that
\begin{align}\label{eq:cond2}
  \sum_i n_i a_{d_i}(f) = 0
\end{align}
for all normalized cusp forms $f\in S_2(4p)$.

We have numerically tested the convergence of the products defining $J_\sD$ for various divisors $\sD$, and our experimental data is consistent with the assertion that conditions \eqref{eq:cond1} and \eqref{eq:cond2} are sufficient for convergence.

\subsection{Overconvergent algorithm} Observe that computing $\Sigma_{d,j}(0,\infty)$ by running over all possible $b$ and $c$ in \eqref{eq: sigma d j zero infty} results in an algorithm of exponential complexity on $j$. That is, it is exponential on the desired $p$-adic accuracy to which $J_d(r,s)$, and hence $J_\sD(r,s)$, are computed. Therefore, this is only feasible for small values of $j$. To compute $J_\sD(r,s)$ we use the same iterative method of \cite{DGL}, which is similar to those of \cite{darmon-vonk}, \cite{darmon-pollack}, or \cite{gmx}. With this method, it is enough to compute $\Sigma_{d,j}(0,\infty)$ for $j=0,1$ to compute $\Phi_{\sD,0}(r,s)$ and $\Phi_{\sD,1}(r,s)$, and then  we obtain $ J_\sD^{N}(r,s) = \prod_{j=0}^{N} \Phi_{\sD,j}(r,s)$ by the formula
\[
  J_\sD^{N}(r,s) =\Phi_{\sD,0}(r,s) \prod_{j=0}^{N-1} U_p^{2j}   (\Phi_{\sD,1}(r,s)) ,
\]
where $U_p$ is the usual $U_p$-operator.
Observe that, by \eqref{eq: Phi d j}, in the definition of $J_\sD$ only matrices whose determinant has even $p$-adic valuation are involved. We can consider also a similar definition but using those with odd $p$-adic valuation of the determinant as follows:
\[
  J_\sD^{N,\mathrm{odd}}(r,s) = \prod_{j=0}^{N-1} U_p^{2j}   (\Phi_{\sD,1/2}(r,s)) .
\]
As we will see in \S\ref{sec: experiments}, both the cocycles $J_\sD$ and $J_\sD^\mathrm{odd}$ seem to produce algebraic quantities when evaluated at special points.

\section{Evaluation at special points}\label{sec: evaluation}
In this section we explain how to evaluate the rigid meromorphic cocycle $J_\sD$ of Section \ref{sec: modular symbol} at arguments coming from certain ATR fields, the  so-called \emph{big ATR points}. This section is inspired by \cite{darmon-logan}, which evaluates Hilbert modular forms at analogous quantities arising from ATR fields.
But first, let us briefly recall how to evaluate rigid meromorphic cocycles at small RM points and small CM points.

\subsection{Evaluation at small RM and CM points}
Let $E$ a real quadratic number field, in which the prime $p$ is inert.
Fix a generator $u$ of the group of norm-one units of $\cO_E$.
By choosing a $\Z$-basis of $\cO_E$, the element $u$ gives rise to a matrix $\gamma_u\in\SL_2(\Z)$.
Since $p$ is inert in $E$, it follows that $\gamma_u$ has exactly two fixed points $\tau_1,\tau_2\in \cH_p$,
which can be computed by solving quadratic equations.
The matrix $\gamma_u$ also acts on the $p$-adic symmetric space $X_p\cong \cH_p \times \cH_p$ via the embedding $\SL_2(\Z)\subseteq \SL_2(K)$.
The four fixed points of this action are given by $(\tau_i,\tau_j)$, $i,j\in\{1,2\}$.
The \emph{reflex field} $E_\tau$ of such a fixed point $\tau=(\tau_i,\tau_j)$ is a quadratic number field contained in the compositum $EK$.
More precisely, it is equal to $E$ if $i=j$, in which case we call $\tau$ a \emph{small RM point}.
Otherwise, $\tau$ is a \emph{small CM point}, and $E_\tau$ is the imaginary quadratic subfield of $EK$ in which $p$ is inert.
In any case, we define the evaluation of $J_\sD$ at $\tau$ to be
\[
  J_\sD[\tau] \coloneq J_\sD(\infty,\gamma_u \infty)(\tau).
\]
One can similarly evaluate the cocycle $J_\sD^\mathrm{odd}$ by means of
\[
  J_\sD^{\mathrm{odd}}[\tau] \coloneq J_\sD^{\mathrm{odd}}(\infty,\gamma_u \infty)(\tau).
\]

\subsection{Evaluation at ATR points}
Let $D>0$ be a fundamental discriminant and let $n$ be a positive integer. Let $F = \Q(\sqrt{D})$ and suppose that $F$ contains an element $\alpha$ of norm $-1$. Define $E = F(\sqrt{n\alpha })$. Since $n\alpha $ has negative norm, the extension $E/F$ is ATR; that is, $E$ has signature $(1,1)$. The field $E$ is not Galois over $\Q$, and if we denote by $\cM$ its Galois closure we have that $\Gal(\cM/\Q)\simeq \mathrm{D_{2\cdot 4}}$, the dihedral group having 8 elements. The field $\mathcal M$ contains $\Q(\sqrt{n\alpha}\sqrt{n\alpha'})$, where $\alpha'$ denotes the $\Gal(F/\Q)$-conjugate of $\alpha$. Since $\alpha\alpha'=-1$, this field is $K=\Q(i)$. It turns out that, in fact, $\cM = E\cdot K$ and that the lattice of subfields of $\cM$ is of the following form:

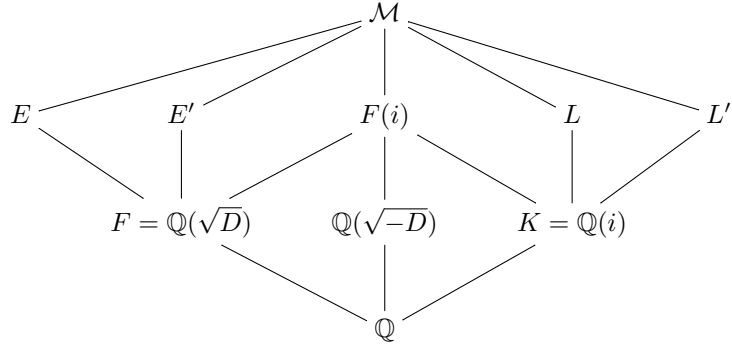
\begin{equation*}
\xymatrix{
         &           &  \cM \ar@{-}[d]\ar@{-}[dll]\ar@{-}[drr] \ar@{-}[dl]\ar@{-}[dr]     &         &           \\
E   & E'   &      F(i)      &    L    &    L'     \\
  &  F=\Q(\sqrt{D})\ar@{-}[ul]\ar@{-}[u]\ar@{-}[ur] & \Q(\sqrt{-D}) \ar@{-}[u] & K=\Q(i)
\ar@{-}[ul]\ar@{-}[u]\ar@{-}[ur]\\
    & &\Q \ar@{-}[ul]\ar@{-}[u]\ar@{-}[ur] & &
}
\end{equation*}
The field $E'$ is conjugate to $E$. The fields $L$ and $L'$ are characterized by being the only quartic non-Galois totally complex fields contained in $\cM$.

The group of units of $E$ whose relative norm to $F$ is equal to $1$ is of rank $1$.
Let $u\in \cO_E$ be a generator of this group, which is unique up to torsion and inversion.
As in the case of small special points, we would like to view $u$ as an element of $\SL_2(K)$.
For this, consider $\tilde{u}\coloneq\No_{\cM/L}(u)$; that is, take the norm from $\cM$ down to $L$. Since $L/K$ is a quadratic extension, by fixing an $\cO_K$-basis of $\cO_L$ the element $\tilde{u}\in L$ gives rise to a matrix $\gamma_u\in \SL_2(\cO_K)$.

The action of $\SL_2(K)$ on $V$ is by twisted conjugation, so a fixed point $x$ for $\gamma_u$ on the $E$-valued points of the quadric $Q$ of isotropic lines is represented by a matrix $M\in V \otimes E$ such that
\begin{align}\label{eq: fixed}
  \gamma_u\cdot M \cdot (\gamma_u')^{-1} = \lambda \cdot M
\end{align}
  for some $\lambda\in E$.
 One can find the matrix $M$ effectively and algorithmically by interpreting \eqref{eq: fixed} as a system of equations where the variables are $\lambda$ and the entries of $M$, together with the extra equation $\det M = 0$.
Since $M$ is a matrix of rank $1$, which is not upper triangular, we can always scale it to have the form 
  \[
M=\left(\begin{array}{rr}
  (\tau_1,\tau_2) & -\tau_1\tau_2\\ 1 &-(\tau_2,\tau_1)
\end{array}\right)
\]
for some $\tau_1,\tau_2\in \C_p$.
We consider pairs of $D$ and $n$ such that the fixed point $\tau\in Q(E)\subseteq Q(\C_p)$ corresponding to $M$ defines a point on $X_p$ or, equivalently, such that $\tau_1,\tau_2$ do not belong to $\Q_p$.
In that case, $\tau\in X_p$ is called a \emph{big ATR point}.
Its \emph{reflex field} $E_\tau$ is the ATR field $E$.
As in the case of small special points, we define the evaluation of $J_\sD$ at the big ATR point $x$ associated to $u$ to be
\begin{align*}
  J_\sD[\tau] \coloneq &J_\sD(\infty,\gamma_u \infty)(\tau_1,\tau_2)\\
	\intertext{and, similarly, we put}
  J_\sD^{\mathrm{odd}}[\tau] \coloneq &J_\sD^{\mathrm{odd}}(\infty,\gamma_u \infty)(\tau_1,\tau_2).
\end{align*}

\subsubsection{Algebraicity conjecture}\label{sec:alg conj}
The main conjecture of \cite{DGL} predicts that the quantity $j=J_\sD[\tau]$, which in principle belongs to $\C_p$, is in fact an algebraic number for every special point $x\in X_p$ as described above.
More precisely, by the conjectural reciprocity law of \cite[Section 4.4]{DGL} it should be contained in the compositum of
\begin{itemize}
\item the field of definition of the rigid meromorphic cocycle, which by Proposition \ref{fieldofdefprop} is equal to $\Q(i)$, and
\item an abelian extension $A$ of $E$, which is unramified outside $2$.
Moreover, let $F_\tau\subseteq E_\tau$ denote the unique subfield of the reflex field of $\tau$ with $[E_\tau : F_\tau]=2$.
Then $A/F_\tau$ is Galois and $\Gal(E_\tau, F_\tau)$ acts on $\Gal(A/E_\tau)$ via inversion.
\end{itemize}
In other words, $E_\tau(j)/E_\tau$ is conjectured to be an abelian extension unramified outside $2$.
Moreover, the only subextensions of the Galois closure of $E_\tau(j)/F_\tau$ that are abelian over $F_\tau$ are polyquadratic.
Since the field of definition of the rigid meromorphic cocycle is $\Q(i)$, we expect $\Q(i)$ to be a subfield of that Galois closure frequently.

\section{Experimental evidence}\label{sec: experiments}
We have computed the quantities $J_\sD[\tau]$ and $J_\sD^{\mathrm{odd}}[\tau]$ for primes $p=5,13,17$ and  for several divisors $\sD$, fundamental discriminants $D$ and positive integers $n$, to a certain $p$-adic precision. For $p=5$ we computed to 300 $p$-adic digits, for $p=13$ to 150 $p$-adic digits and for $p=17$ to 130 $p$-adic digits. The code used to perform the calculations can be found in the repository \cite{Masdeu}.

The complete computational results are too extensive to be included in full in this article. A comprehensive table containing all data is available online in \cite{Masdeu2}. Here, we present only a summary of the most relevant findings together with partial tables that can be found in Appendix \ref{appendix}. In the table of \cite{Masdeu2}, the column \texttt{label} refers to the divisor $\sD$ and the parity of the cocycle; for example the label \texttt{6\_24\_even} corresponds to the cocycle $J_\sD$ with $\sD=(6)-(24)$, the label \texttt{6·6\_14\_odd} to the cocycle $J_\sD^{\mathrm{odd}}$ with $\sD = 2 (6) -(14)$, the label \texttt{3·3·15\_21\_odd} corresponds $J_\sD^\mathrm{odd}$ with  $\sD=2( 3) + (15)-(21)$, and so on. To lighten the notation, from now on we will only speak of the cocycles of the form $J_\sD$, with the understanding that identical computations have been carried out for those of the form $J_\sD^\mathrm{odd}$.

The column labeled as \texttt{type} indicates whether the computed value corresponds to a big ATR point, a small CM point or a small RM point. Additional columns display the values of $D$ and $n$. In the case of big ATR points, these are the parameters $D$ and $n$ defined in Section \ref{sec: evaluation}; for small special points, $D$ is the absolute value of the discriminant of the reflex field and $n$ is not relevant for these cases (and is set to 1 in the table).  Recall that we have chosen $p$ to be split in $K$, say as $p = \fp_1\fp_2$. 
In order to simplify the computations, we only considered parameters of $p$, $D$, and $n$ for which $\fp_1$ and $\fp_2$ are unramified in $L$, so that $L_{\fp_i}$ is the quadratic unramified extension of $\Q_p$.

Instead of displaying the computed value  $J_\sD[\tau]$, what is shown in the table are the two related quantities
 \begin{align*}
   J_\sD[\tau]^{\mathrm{triv}} \coloneq {J_\sD[\tau]}\cdot{\overline{J_\sD[\tau]}}\quad \mbox{and}\quad    J_\sD[\tau]^{\mathrm{conj}} \coloneq \frac{J_\sD[\tau]}{\overline{J_\sD[\tau]}},
 \end{align*}
 where the bar denotes the non-trivial conjugation of $L_{\fp_i}/\Q_p$. The column labeled as \texttt{J} indicates then either the value of $J_\sD[\tau]^{\mathrm{triv}}$ or $J_\sD[\tau]^{\mathrm{conj}}$ (depending on the value of the column labelled as \texttt{char}), coded as an integer.

 The column  \texttt{trivial} displays whether the computed value $J_\sD[\tau]^{\bullet}$ (for $\bullet \in \{\mathrm{triv},\mathrm{conj}\}$) is equal to $1$ up to the working precision. The column \texttt{recognized} is set to \texttt{True} whenever we have been able to recognize $J_\sD[\tau]^{\bullet}$ as a probable algebraic integer. This means that we have found a polynomial $P(x)\in\Q[x]$ such that $P(J_\sD[\tau]^{\bullet})=0$ up to the working $p$-adic precision, and in this case $P(x)$ is displayed in the column \texttt{poly}. In other words,  $J_\sD[\tau]^{\bullet}$ is $p$-adically close to an algebraic element $J^{\mathrm{alg}}$. The column \texttt{field} displays a defining polynomial for the field $\Q(J^{\mathrm{alg}})$, while the column \texttt{factor} displays the factorization of the norm of $J^{\mathrm{alg}}$ down to $\Q$.

 Of course, there always exist many rational polynomials $P(x)\in\Q[x]$ such that $P(J_\sD[\tau]^{\bullet})=0$ up to the working $p$-adic precision. We consider $J_\sD[\tau]^{\bullet}$ to be \emph{recognized} as algebraic if either
 \begin{enumerate}[(i)]
 \item  we find a polynomial with very small height compared to the working precision, or
   \item  we find a polynomial with large height but such that the support of the rational number $\mathrm{Nm}_{\Q(J^{\mathrm{alg}})/\Q}(J^{\mathrm{alg}})$ consists of a small number of small primes.
 \end{enumerate}
These cases are indicated in the tables as \texttt{(a)} and \texttt{(l)}, respectively.

In Appendix \ref{appendix} we include the data corresponding to the fields generated by the values $J^\mathrm{alg}$ and their factorization  for the 190 recognized big ATR special values. Table~\ref{tbl:table-p5} presents the data for the prime $p=5$, Table \ref{tbl:table-p13} for $p=13$, and Table \ref{tbl:table-p17} for $p=17$.

Examining the data in table \cite{Masdeu2},  we observe that the quantities $J_\sD[\tau]^{\mathrm{triv}}$ and $J_\sD[\tau]^{\mathrm{odd,triv}}$ are equal to $1$ (up to the working precision) in most cases, though not in all. There are 13 instances where $J_\sD[\tau]^{\mathrm{triv}}$ is non-trivial and 13 where $J_\sD[\tau]^{\mathrm{odd,triv}}$ is non-trivial; we have not been able to recognize any of these non-trivial quantities as algebraic numbers. The quantities $J_\sD[\tau]^{\mathrm{conj}}$ and $J_\sD[\tau]^{\mathrm{odd,conj}}$, on the other hand, are different from 1 in many cases. Among these, in 190 instances we have been able to identify the non-trivial value as a probable algebraic integer $J^{\mathrm{alg}}$, and partial data for these cases is displayed in Tables \ref{tbl:table-p5}, \ref{tbl:table-p13} and \ref{tbl:table-p17} in Appendix \ref{appendix}.

\begin{obs}\label{obs: BY}
  Interestingly enough, in all the examples that we found, the primes $\ell$ in the support of $\mathrm{Nm}_{\Q(J^{\mathrm{alg}})/\Q}(J^{\mathrm{alg}})$ satisfy the following condition:
  \begin{equation}\label{eq: cond l}
    \begin{split}
      & 4\ell p\mid s^2q-n^2 \text{ for some } s\in \mathrm{supp}(\sD)  \text{ and some integer  } |n|<s\sqrt{q},\\ & \text{ where } q = \mathrm{Nm}_{\Q(i)/\Q}(d_{L/\Q(i)}) \text{ and } d_{L/\Q(i)}\text{ is the discriminant of }  L/\Q(i).
 \end{split}
  \end{equation}
    This condition can be interpreted as a $p$-adic analogue of \cite[Corollary 1.3]{BY} for the primes appearing in the support of CM values of Hilbert modular forms defined over real quadratic fields. In fact, condition \eqref{eq: cond l} is formally the same as \cite[Corollary 1.3]{BY} if we replace (in the notation of \textit{loc.cit.}) the values $m$ such that $\tilde{c}(-m)=0$ by the integers $s$ in the support of $\sD$, the CM field by $L$, and the discriminant of the real quadratic base field by the prime $p$.
\end{obs}

\begin{ex}
  Consider the divisor $\sD= 2(3) + (15) - (21)$, the prime $p=13$, discriminant $D=673$, and $n=3$. The field $E$ has defining polynomial $x^4 - 69x^2 - 324$, and the field $\cM$ has defining polynomial
  \begin{align*}
    x^8 - 4x^7 + 8x^6 + 62x^5 + 3505x^4 - 7142x^3 + 2450x^2 - 84560x + 1459264.
  \end{align*}
  The computed special value $J_\sD[\tau]^\mathrm{conj}$ is a root modulo $13^{150}$ of the polynomial $x^4 + \frac{a}{d}x^3+\frac{b}{d}x^2+\frac{a}{d}x+1$, where 
  \begin{small}
  \begin{align*}
    a =  -&211349500654446599836316470733755251435519541615714511636\\& 821171012346603970976458836245074566482561365850918592452,
  \end{align*}
  \begin{align*}
    b= &286782751017446189391219456317914883658028399549685346229\\& 776145380151720639238579273506832330717517876949503918630,
  \end{align*}
  \begin{align*}
    d = &828851276413836168532193777337568257582850330951414053466\\& 75932530059510450523156131010624045718836312406000008881.
  \end{align*}
\end{small}
If we denote by $j$ a root of this polynomial, it turns out that the field $\Q(j)$ is a quartic field of discriminant $2^{2} \cdot 3 \cdot 673^{2}$ and minimal polynomial $x^4 + 49x^2 + 432$. The extension $E(j)$ that $j$ generates over $E$ is such that $[E(j)\colon E]=2$ and its relative discriminant has norm 16. Therefore, $E(j)/E$ is unramified outside $2$.
Moreover, if we take compositum with $K=\Q(i)$ we obtain an unramified extension, that is, $E(j)\cdot K/E\cdot K$ is unramified. Finally, the norm $\mathrm{Norm}_{\Q(j)/\Q}(j)$ is a rational number supported at the primes
\begin{align*}
   659, 1297, 1579, 6007, 9631, 12791, 13099,
\end{align*}
all of them satisfying condition \eqref{eq: cond l}. In fact, in this example there are 101 possible primes $\ell$ satisfying  \eqref{eq: cond l}. Of these, there are 19 such that  all primes $\mathfrak{l}$ of $L$ dividing $\ell$ are fixed by the non-trivial automorphism of $L_{\fp_i}/\Q_p$. Since $\mathrm{Norm}_{L_{\fp_i}/\QQ_p} (J_\sD[\tau]^{\mathrm{conj}})=1$, these primes will never appear and therefore there are only 81 possibilities for the primes $\ell$ dividing $\mathrm{Norm}_{\QQ(j)/\QQ} (j)$. The largest prime allowed by condition  \eqref{eq: cond l} is 25447, and in fact there are only 13 primes larger that 10000 that satisfy the condition; the only ones appearing are 12791 and 13099.
\end{ex}

\subsection*{Tables of small CM and small RM points} In addition to computing the values at big ATR points, we have also evaluated all cocycles at several small CM and small RM points, and we have recognized many of these values as probably algebraic. Although the first such experiments were already carried out in \cite{DGL}, our computations provide further numerical evidence for the rationality conjecture in these cases, as we have found 56 new probably algebraic non-trivial small CM points and 591 small RM points. The complete data for these computations can be found in table \cite{Masdeu2}, by selecting \texttt{small CM} or \texttt{small RM} in the \texttt{type} column.

For the case of small RM, as we explained in the introduction, they are related to the special values of the rigid meromorphic cocycles of Darmon--Vonk \cite{darmon-vonk}. However, the precise relation between these two types of rigid meromorphic cocycles and the respective special points has not yet been formulated, and we believe that our tables may serve as useful experimental data towards determining this exact relation.

As for the case of small CM points, perhaps the most remarkable novelty is that we have found some special values defined over number fields strictly larger than $\Q(i)$, whereas in \cite{DGL} all such values were defined over $\Q(i)$.
\begin{ex}
For $\sD = 2(3)+15-21$, $p=13$, and $D = 41$, the corresponding special value generates the quartic field $M$ with defining polynomial $f=x^{4} - 2 x^{3} - x^{2} + 2 x + 2$. This field has discriminant $2^{4} \cdot 41$ and it contains $\Q(i)$. Moreover, if we denote $L=\Q(i,\sqrt{-41})$, the extension $M L/M$ is quadratic and unramified.
\end{ex}

Out of the 56 probably algebraic values corresponding to small CM points that we have found, in 9 cases the situation is similar to the one described in the previous example:
the value of the small CM point generates a quartic field $M$ that is unramified outside the primes dividing $2 D$
and contains $\Q(i)$. Moreover, for $L=\Q(i,\sqrt{-D})$ the extension $ML/M$ is quadratic and unramified.
There are another 9 cases in which the special CM value generates a quadratic field $M$ different from $\Q(i)$,
and in these cases as well the extension $ML/M$ is unramified. Finally, in 29 cases the special value generates $\QQ(i)$, and in 9 cases it belongs to $\Q$.

\begin{obs}
As a final piece of evidence, we checked that all big ATR and small CM points we identified are defined over fields that match the prediction of the algebraicity conjecture in Section \ref{sec:alg conj}.
\end{obs}

\begin{landscape}
\appendix

\section{Tables}\label{appendix}

  {\tiny
\begin{longtable}{p{5em}p{3em}p{.5em}p{7cm}p{6cm}}
\caption{Recognized points of type bigATR for p = 5} \label{tbl:table-p5} \\
\toprule
label & D & n & field & factor \\
\midrule
\endfirsthead
\caption[]{Recognized points of type bigATR for p = 5} \\
\toprule
label & D & n & field & factor \\
\midrule
\endhead
\midrule
\multicolumn{5}{r}{Continued on next page} \\
\midrule
\endfoot
\bottomrule
\endlastfoot
6\_12\_even & $29$ & $3$ & $x^{8} - 2 x^{7} + 2 x^{6} - 30 x^{5} + 254 x^{4} - 870 x^{3} + 1682 x^{2} - 1682 x + 841$ & $J = 7^{12} \cdot 23^{-8} \cdot 67^{-4} \cdot 283^{-4} \cdot 631^{4} \cdot 1543^{-4} $ (l) \\
6·6\_14\_even & $29$ & $3$ & $x^{4} - 15 x^{2} - 9$ & $J = 7^{4} \cdot 107^{-8} \cdot 347^{8} \cdot 631^{8} \cdot 863^{-4} \cdot 1031^{4} \cdot 1327^{-4} \cdot 1543^{-8} $ (l) \\
6\_12\_odd & $29$ & $3$ & $x^{8} - 2 x^{7} + 2 x^{6} - 30 x^{5} + 254 x^{4} - 870 x^{3} + 1682 x^{2} - 1682 x + 841$ & $J = 7^{8} \cdot 71^{8} \cdot 1399^{-4} \cdot 2203^{-4} $ (l) \\
6\_24\_even & $41$ & $1$ & $x^{8} + 13 x^{6} + 40 x^{4} + 52 x^{2} + 16$ & $J = 2^{-48} \cdot 103^{-4} \cdot 127^{4} $ (l) \\
6\_24\_even & $61$ & $1$ & $x^{8} - 4 x^{7} + 8 x^{6} + 2 x^{5} + 27 x^{4} - 66 x^{3} + 50 x^{2} - 90 x + 81$ & $J = 3^{16} \cdot 19^{16} \cdot 47^{-8} \cdot 131^{-8} \cdot 199^{-8} \cdot 271^{-4} \cdot 463^{-8} \cdot 2467^{-4} \cdot 4651^{4} $ (l) \\
6\_12\_even & $61$ & $1$ & $x^{8} - 4 x^{7} + 8 x^{6} + 2 x^{5} + 27 x^{4} - 66 x^{3} + 50 x^{2} - 90 x + 81$ & $J = 3^{-8} \cdot 19^{8} \cdot 127^{-8} \cdot 271^{-4} \cdot 379^{-4} \cdot 571^{4} \cdot 1483^{4} $ (l) \\
6\_24\_odd & $61$ & $1$ & $x^{8} - 4 x^{7} + 8 x^{6} + 2 x^{5} + 27 x^{4} - 66 x^{3} + 50 x^{2} - 90 x + 81$ & $J = 3^{16} \cdot 19^{-4} \cdot 103^{4} \cdot 367^{-4} \cdot 439^{-4} \cdot 487^{-8} \cdot 619^{-4} \cdot 3163^{-4} \cdot 4099^{4} \cdot 4987^{4} \cdot 5779^{4} \cdot 6691^{4} $ (l) \\
6\_12\_odd & $61$ & $1$ & $x^{8} - 4 x^{7} + 8 x^{6} + 2 x^{5} + 27 x^{4} - 66 x^{3} + 50 x^{2} - 90 x + 81$ & $J = 3^{8} \cdot 103^{4} \cdot 367^{-4} \cdot 439^{-4} \cdot 691^{-4} \cdot 859^{4} \cdot 1723^{-4} $ (l) \\
6\_12\_even & $89$ & $3$ & $x^{8} + 15 x^{6} + 191 x^{4} + 1845 x^{2} + 1156$ & $J = 2^{24} \cdot 47^{-4} \cdot 67^{4} $ (l) \\
6\_12\_odd & $89$ & $2$ & $x^{8} + 57 x^{4} + 256$ & $J = 47^{4} \cdot 523^{4} \cdot 1867^{-4} \cdot 2011^{-2} \cdot 5323^{-2} \cdot 9067^{-2} $ (l) \\
6\_24\_even & $109$ & $1$ & $x^{8} - 4 x^{7} + 8 x^{6} + 10 x^{5} + 63 x^{4} - 154 x^{3} + 162 x^{2} - 126 x + 49$ & $J = 3^{-16} \cdot 7^{8} \cdot 43^{8} \cdot 83^{8} \cdot 263^{8} \cdot 431^{8} \cdot 727^{-4} \cdot 7243^{-4} \cdot 8467^{-4} \cdot 8803^{4} $ (l) \\
6\_12\_even & $109$ & $1$ & $x^{8} - 4 x^{7} + 8 x^{6} + 10 x^{5} + 63 x^{4} - 154 x^{3} + 162 x^{2} - 126 x + 49$ & $J = 3^{-8} \cdot 31^{-8} \cdot 43^{-8} \cdot 211^{-4} \cdot 307^{-4} \cdot 727^{-4} $ (l) \\
6\_24\_odd & $109$ & $1$ & $x^{8} - 4 x^{7} + 8 x^{6} + 10 x^{5} + 63 x^{4} - 154 x^{3} + 162 x^{2} - 126 x + 49$ & $J = 3^{-16} \cdot 31^{-4} \cdot 223^{-4} \cdot 311^{8} \cdot 751^{4} \cdot 1699^{-4} \cdot 2251^{-4} \cdot 5107^{4} \cdot 5563^{4} \cdot 6571^{-4} \cdot 10267^{-4} \cdot 11491^{4} $ (l) \\
6\_12\_odd & $109$ & $1$ & $x^{8} - 4 x^{7} + 8 x^{6} + 10 x^{5} + 63 x^{4} - 154 x^{3} + 162 x^{2} - 126 x + 49$ & $J = 3^{8} \cdot 7^{8} \cdot 31^{-4} \cdot 223^{-4} \cdot 751^{4} \cdot 1483^{4} \cdot 2131^{4} \cdot 2659^{-4} $ (l) \\
6·6\_14\_odd & $149$ & $1$ & $x^{4} - 7 x^{2} - 25$ & $J = 7^{-16} \cdot 31^{8} \cdot 103^{-8} \cdot 107^{8} \cdot 631^{8} \cdot 911^{-4} \cdot 1039^{8} \cdot 1063^{-8} \cdot 2087^{-4} \cdot 3719^{-4} \cdot 4463^{-4} $ (l) \\
6\_12\_even & $149$ & $1$ & $x^{8} - 4 x^{7} + 8 x^{6} + 10 x^{5} + 83 x^{4} - 194 x^{3} + 162 x^{2} - 306 x + 289$ & $J = 7^{-16} \cdot 19^{-8} \cdot 47^{8} \cdot 179^{8} \cdot 967^{4} \cdot 1459^{4} \cdot 2251^{4} \cdot 3547^{4} $ (l) \\
6·6\_14\_even & $149$ & $1$ & $x^{4} - 7 x^{2} - 25$ & $J = 7^{-8} \cdot 19^{-24} \cdot 47^{16} \cdot 67^{-8} \cdot 191^{-4} \cdot 263^{-4} \cdot 967^{8} \cdot 4111^{-4} \cdot 4327^{-4} \cdot 4943^{4} \cdot 5279^{-4} \cdot 5399^{4} $ (l) \\
6\_12\_odd & $149$ & $1$ & $x^{8} - 4 x^{7} + 8 x^{6} + 10 x^{5} + 83 x^{4} - 194 x^{3} + 162 x^{2} - 306 x + 289$ & $J = 7^{-8} \cdot 67^{4} \cdot 103^{-8} \cdot 251^{-8} \cdot 631^{4} \cdot 859^{-4} \cdot 1039^{4} \cdot 1063^{-4} \cdot 2707^{4} $ (l) \\
6\_12\_odd & $269$ & $1$ & $x^{8} - 4 x^{7} + 8 x^{6} + 10 x^{5} + 143 x^{4} - 314 x^{3} + 162 x^{2} - 846 x + 2209$ & $J = 43^{-4} \cdot 47^{-8} \cdot 131^{-8} \cdot 1039^{4} \cdot 3187^{-4} \cdot 3307^{-4} \cdot 6091^{4} $ (l) \\
6·6\_14\_odd & $389$ & $1$ & $x^{4} - 17 x^{2} - 25$ & $J = 7^{28} \cdot 11^{8} \cdot 19^{8} \cdot 59^{8} \cdot 919^{-8} \cdot 991^{-4} \cdot 1423^{8} \cdot 2447^{4} \cdot 4391^{4} \cdot 4943^{-4} \cdot 6679^{4} \cdot 10567^{-4} \cdot 12959^{-4} $ (l) \\
6\_12\_even & $389$ & $1$ & $x^{8} - 4 x^{7} + 8 x^{6} + 10 x^{5} + 203 x^{4} - 434 x^{3} + 162 x^{2} - 1386 x + 5929$ & $J = 7^{-16} \cdot 19^{8} \cdot 2239^{4} \cdot 2467^{-4} \cdot 2767^{-4} \cdot 2791^{-4} \cdot 3907^{-4} \cdot 6763^{4} \cdot 9547^{4} \cdot 9619^{4} $ (l) \\
6·6\_14\_even & $389$ & $1$ & $x^{4} - 17 x^{2} - 25$ & $J = 7^{24} \cdot 11^{-8} \cdot 19^{-8} \cdot 79^{8} \cdot 599^{-4} \cdot 2039^{-4} \cdot 2239^{-8} \cdot 2767^{-8} \cdot 2791^{-8} \cdot 5303^{4} \cdot 7487^{-4} \cdot 11159^{-4} \cdot 12511^{4} \cdot 13127^{-4} \cdot 13367^{4} $ (l) \\
6\_12\_odd & $389$ & $1$ & $x^{8} - 4 x^{7} + 8 x^{6} + 10 x^{5} + 203 x^{4} - 434 x^{3} + 162 x^{2} - 1386 x + 5929$ & $J = 7^{-12} \cdot 67^{4} \cdot 179^{8} \cdot 331^{4} \cdot 919^{4} \cdot 1423^{4} \cdot 4651^{4} \cdot 10459^{4} $ (l) \\
6\_12\_odd & $401$ & $3$ & $x^{8} - 4 x^{7} + 8 x^{6} + 110 x^{5} + 3313 x^{4} - 6854 x^{3} + 6962 x^{2} - 3776 x + 1024$ & $J^5 = 7^{8} \cdot 383^{4} $ (l) \\
6\_24\_even & $409$ & $1$ & $x^{8} - 4 x^{7} + 8 x^{6} + 30 x^{5} + 313 x^{4} - 694 x^{3} + 722 x^{2} - 456 x + 144$ & $J = 2^{32} \cdot 3^{-8} \cdot 71^{8} \cdot 1879^{4} \cdot 2383^{4} $ (l) \\
6\_24\_odd & $409$ & $1$ & $x^{8} - 4 x^{7} + 8 x^{6} + 30 x^{5} + 313 x^{4} - 694 x^{3} + 722 x^{2} - 456 x + 144$ & $J = 2^{-16} \cdot 3^{-8} \cdot 83^{8} \cdot 103^{-4} \cdot 167^{8} \cdot 823^{4} \cdot 1063^{-4} \cdot 1567^{-4} \cdot 1759^{4} \cdot 2671^{-4} $ (l) \\
6·6\_14\_odd & $421$ & $1$ & $x^{4} - 15 x^{2} - 49$ & $J = 3^{24} \cdot 7^{-4} \cdot 31^{16} \cdot 79^{4} \cdot 167^{8} \cdot 1447^{8} \cdot 1607^{4} \cdot 2551^{8} \cdot 2879^{4} \cdot 6911^{4} \cdot 8423^{-4} \cdot 11447^{-4} $ (l) \\
6\_12\_even & $421$ & $1$ & $x^{8} - 4 x^{7} + 8 x^{6} + 18 x^{5} + 247 x^{4} - 538 x^{3} + 338 x^{2} - 1638 x + 3969$ & $J = 3^{-8} \cdot 7^{8} \cdot 31^{8} \cdot 103^{4} \cdot 107^{-8} \cdot 199^{-4} \cdot 991^{-4} \cdot 1783^{-4} \cdot 2287^{-4} \cdot 6547^{-4} $ (l) \\
6·6\_14\_even & $421$ & $1$ & $x^{4} - 15 x^{2} - 49$ & $J = 3^{-32} \cdot 11^{-24} \cdot 103^{8} \cdot 139^{-8} \cdot 199^{-8} \cdot 863^{4} \cdot 991^{-8} \cdot 1327^{-4} \cdot 1783^{-8} \cdot 1831^{4} \cdot 2287^{-8} $ (l) \\
6\_12\_odd & $421$ & $1$ & $x^{8} - 4 x^{7} + 8 x^{6} + 18 x^{5} + 247 x^{4} - 538 x^{3} + 338 x^{2} - 1638 x + 3969$ & $J = 3^{8} \cdot 7^{12} \cdot 11^{-16} \cdot 31^{8} \cdot 139^{4} \cdot 1447^{4} \cdot 2179^{4} \cdot 2551^{4} \cdot 2707^{-4} \cdot 4363^{-4} \cdot 10243^{-4} $ (l) \\
6\_24\_even & $569$ & $1$ & $x^{8} - 4 x^{7} + 8 x^{6} + 30 x^{5} + 393 x^{4} - 854 x^{3} + 722 x^{2} - 1976 x + 2704$ & $J = 2^{-80} \cdot 67^{-8} \cdot 107^{-8} \cdot 139^{8} \cdot 239^{8} $ (l) \\
6\_24\_odd & $569$ & $1$ & $x^{8} - 4 x^{7} + 8 x^{6} + 30 x^{5} + 393 x^{4} - 854 x^{3} + 722 x^{2} - 1976 x + 2704$ & $J = 2^{16} \cdot 7^{28} \cdot 79^{4} \cdot 307^{-8} \cdot 2719^{-4} \cdot 3727^{4} $ (l) \\
6\_12\_even & $661$ & $1$ & $x^{8} - 4 x^{7} + 8 x^{6} + 2 x^{5} + 327 x^{4} - 666 x^{3} + 50 x^{2} - 1590 x + 25281$ & $J = 3^{24} \cdot 11^{-16} \cdot 31^{-8} \cdot 43^{-8} \cdot 239^{-8} \cdot 311^{-8} \cdot 359^{8} \cdot 739^{4} \cdot 1831^{-4} \cdot 1867^{4} \cdot 2719^{-4} \cdot 3511^{4} \cdot 4423^{-4} \cdot 4759^{-4} \cdot 8179^{4} \cdot 8731^{4} \cdot 11587^{-4} \cdot 12043^{4} \cdot 16747^{-4} $ (l) \\
6·6\_14\_even & $661$ & $1$ & $x^{4} - 25 x^{2} - 9$ & $J = 3^{80} \cdot 11^{-8} \cdot 43^{-16} \cdot 59^{-8} \cdot 239^{-16} \cdot 359^{-4} \cdot 1831^{-8} \cdot 2719^{-8} \cdot 2767^{-4} \cdot 3511^{8} \cdot 4423^{-8} \cdot 4759^{-8} \cdot 20063^{-4} $ (l) \\
6\_12\_odd & $661$ & $1$ & $x^{8} - 4 x^{7} + 8 x^{6} + 2 x^{5} + 327 x^{4} - 666 x^{3} + 50 x^{2} - 1590 x + 25281$ & $J = 3^{-24} \cdot 31^{4} \cdot 43^{-4} \cdot 127^{4} \cdot 199^{4} \cdot 1327^{-4} \cdot 2383^{-4} \cdot 3019^{-4} \cdot 3691^{-4} \cdot 4567^{4} \cdot 4591^{-4} \cdot 5827^{-4} \cdot 9091^{-4} \cdot 9619^{4} \cdot 13723^{4} \cdot 14107^{4} \cdot 14947^{4} \cdot 17659^{4} \cdot 17851^{-4} $ (l) \\
6\_12\_even & $701$ & $1$ & $x^{8} - 4 x^{7} + 8 x^{6} + 42 x^{5} + 567 x^{4} - 1226 x^{3} + 1250 x^{2} - 950 x + 361$ & $J = 7^{12} \cdot 19^{-8} \cdot 83^{8} \cdot 103^{4} \cdot 163^{-4} \cdot 199^{-4} \cdot 3019^{-4} \cdot 9883^{4} \cdot 12739^{-4} \cdot 16651^{4} $ (l) \\
6\_12\_odd & $701$ & $1$ & $x^{8} - 4 x^{7} + 8 x^{6} + 42 x^{5} + 567 x^{4} - 1226 x^{3} + 1250 x^{2} - 950 x + 361$ & $J = 7^{-8} \cdot 43^{-4} \cdot 487^{4} \cdot 607^{4} \cdot 2671^{-4} \cdot 3391^{-4} \cdot 3463^{4} \cdot 4567^{-4} \cdot 5023^{4} \cdot 10243^{-4} \cdot 10771^{4} \cdot 15259^{-4} $ (l) \\
6·6\_14\_odd & $821$ & $1$ & $x^{4} - 25 x^{2} - 49$ & $J = 7^{-8} \cdot 23^{16} \cdot 103^{-8} \cdot 223^{-8} \cdot 263^{-8} \cdot 271^{-8} \cdot 727^{8} \cdot 1607^{-4} \cdot 3079^{-8} \cdot 4663^{-8} \cdot 4951^{-4} \cdot 5167^{8} \cdot 5839^{8} \cdot 7823^{-4} \cdot 10271^{-4} \cdot 19583^{4} \cdot 23279^{-4} \cdot 23447^{4} \cdot 27743^{4} \cdot 28751^{-4} \cdot 31231^{4} \cdot 31663^{4} \cdot 31991^{-4} $ (l) \\
6\_12\_even & $821$ & $1$ & $x^{8} - 4 x^{7} + 8 x^{6} + 18 x^{5} + 447 x^{4} - 938 x^{3} + 338 x^{2} - 4238 x + 26569$ & $J = 7^{4} \cdot 19^{-8} \cdot 23^{-8} \cdot 43^{8} \cdot 103^{-8} \cdot 199^{-4} \cdot 211^{8} \cdot 1471^{4} \cdot 3547^{-4} \cdot 4327^{-4} \cdot 4903^{-4} \cdot 5431^{-4} \cdot 5743^{4} \cdot 9811^{-4} \cdot 19891^{-4} \cdot 21523^{-4} \cdot 23371^{-4} $ (l) \\
6·6\_14\_even & $821$ & $1$ & $x^{4} - 25 x^{2} - 49$ & $J = 7^{-20} \cdot 23^{8} \cdot 43^{8} \cdot 67^{-8} \cdot 79^{-8} \cdot 163^{8} \cdot 199^{8} \cdot 211^{8} \cdot 463^{4} \cdot 1031^{4} \cdot 1471^{-8} \cdot 4327^{8} \cdot 4903^{-8} \cdot 5431^{8} \cdot 5743^{8} \cdot 9199^{4} \cdot 11831^{-4} \cdot 14303^{4} \cdot 20759^{-4} \cdot 24103^{-4} \cdot 27127^{-4} \cdot 28319^{4} \cdot 30223^{4} \cdot 30871^{4} \cdot 31847^{-4} \cdot 32183^{-4} $ (l) \\
6\_12\_odd & $821$ & $1$ & $x^{8} - 4 x^{7} + 8 x^{6} + 18 x^{5} + 447 x^{4} - 938 x^{3} + 338 x^{2} - 4238 x + 26569$ & $J = 7^{8} \cdot 23^{-8} \cdot 43^{-4} \cdot 67^{8} \cdot 71^{-8} \cdot 79^{-4} \cdot 223^{-4} \cdot 727^{4} \cdot 3079^{4} \cdot 4051^{-4} \cdot 4663^{-4} \cdot 5167^{4} \cdot 5839^{4} \cdot 11443^{-4} \cdot 13339^{4} \cdot 16651^{4} \cdot 17659^{-4} $ (l) \\
6\_12\_odd & $881$ & $1$ & $x^{8} - 4 x^{7} + 8 x^{6} + 22 x^{5} + 497 x^{4} - 1046 x^{3} + 450 x^{2} - 4920 x + 26896$ & $J = 2^{-144} \cdot 43^{8} $ (l) \\
6·6\_14\_odd & $1069$ & $1$ & $x^{4} - 13 x^{2} - 225$ & $J = 3^{24} \cdot 19^{-8} \cdot 59^{16} \cdot 79^{-16} \cdot 211^{-8} \cdot 227^{-8} \cdot 2423^{4} \cdot 3119^{-4} \cdot 4159^{8} \cdot 4447^{-4} \cdot 5407^{-8} \cdot 7127^{-4} \cdot 7687^{-8} \cdot 15551^{4} \cdot 30671^{4} $ (l) \\
6·6\_14\_even & $1069$ & $1$ & $x^{4} - 13 x^{2} - 225$ & $J = 3^{48} \cdot 67^{8} \cdot 71^{-4} \cdot 223^{8} \cdot 239^{4} \cdot 431^{4} \cdot 863^{-4} \cdot 1607^{-4} \cdot 1663^{4} \cdot 2383^{-8} \cdot 2767^{-8} \cdot 3511^{-4} \cdot 3607^{-8} \cdot 3943^{-8} \cdot 7591^{-8} \cdot 7639^{8} \cdot 19727^{4} \cdot 27647^{4} \cdot 32831^{4} \cdot 37223^{4} $ (l) \\
6\_12\_odd & $1129$ & $1$ & $x^{8} - 4 x^{7} + 8 x^{6} + 30 x^{5} + 673 x^{4} - 1414 x^{3} + 722 x^{2} - 7296 x + 36864$ & $J^9 = 47^{-8} \cdot 127^{8} \cdot 251^{-8} $ (l) \\
6\_12\_odd & $1241$ & $1$ & $x^{8} - 4 x^{7} + 8 x^{6} - 2 x^{5} + 617 x^{4} - 1238 x^{3} + 18 x^{2} - 1848 x + 94864$ & $J = 2^{32} \cdot 7^{4} \cdot 107^{-4} \cdot 131^{-4} $ (l) \\
6·6\_14\_odd & $1721$ & $3$ & $x^{4} - 33 x^{2} - 3600$ & $J = 2^{-32} \cdot 11^{48} \cdot 19^{-16} \cdot 79^{8} \cdot 83^{-16} \cdot 283^{-16} \cdot 379^{-8} \cdot 751^{-8} \cdot 947^{8} \cdot 1031^{8} \cdot 2347^{-8} \cdot 3359^{-16} \cdot 3803^{8} \cdot 6967^{-8} $ (l) \\
6·6\_14\_odd & $1889$ & $1$ & $x^{4} - 17 x^{2} - 400$ & $J = 2^{336} \cdot 71^{-8} \cdot 79^{4} \cdot 83^{8} \cdot 191^{4} \cdot 227^{8} \cdot 239^{-4} \cdot 547^{4} \cdot 1831^{-4} $ (l) \\
6·6\_14\_odd & $1901$ & $1$ & $x^{4} - 35 x^{2} - 169$ & $J^3 = 7^{-4} \cdot 11^{-32} \cdot 19^{24} \cdot 31^{8} \cdot 43^{8} \cdot 79^{24} \cdot 83^{-8} \cdot 251^{16} \cdot 271^{8} \cdot 491^{8} \cdot 907^{-8} \cdot 1427^{8} \cdot 2207^{-4} \cdot 3919^{-8} \cdot 4663^{4} \cdot 5171^{-8} \cdot 5659^{-8} \cdot 9311^{4} \cdot 9767^{-4} \cdot 13687^{-8} \cdot 13799^{4} \cdot 18127^{4} \cdot 43943^{4} \cdot 47287^{4} \cdot 50159^{-4} \cdot 62119^{-4} \cdot 69959^{4} \cdot 70079^{4} \cdot 71191^{-4} $ (l) \\
6·6\_14\_even & $1901$ & $1$ & $x^{4} - 35 x^{2} - 169$ & $J^3 = 7^{-16} \cdot 11^{8} \cdot 19^{-24} \cdot 31^{16} \cdot 43^{8} \cdot 79^{-12} \cdot 83^{8} \cdot 167^{8} \cdot 227^{-8} \cdot 463^{4} \cdot 599^{-4} \cdot 683^{8} \cdot 1051^{16} \cdot 1087^{-8} \cdot 1439^{-4} \cdot 1579^{8} \cdot 3499^{-8} \cdot 4783^{8} \cdot 4951^{8} \cdot 5507^{-8} \cdot 9127^{-8} \cdot 11311^{8} \cdot 20231^{4} \cdot 28631^{4} \cdot 42359^{4} \cdot 51263^{-4} \cdot 54167^{-4} \cdot 58727^{-4} \cdot 62903^{-4} \cdot 64927^{4} \cdot 67967^{4} \cdot 68111^{-4} \cdot 69463^{-4} \cdot 72559^{4} $ (l) \\
6\_12\_even & $2521$ & $1$ & $x^{8} - 4 x^{7} + 8 x^{6} + 62 x^{5} + 1737 x^{4} - 3606 x^{3} + 2450 x^{2} - 22680 x + 104976$ & $J = 2^{-176} \cdot 283^{8} \cdot 419^{-8} $ (l) \\
6\_12\_odd & $2521$ & $1$ & $x^{8} - 4 x^{7} + 8 x^{6} + 62 x^{5} + 1737 x^{4} - 3606 x^{3} + 2450 x^{2} - 22680 x + 104976$ & $J = 2^{96} \cdot 3^{-16} \cdot 7^{-8} \cdot 151^{-8} $ (l) \\
\end{longtable}

}
{\tiny
\begin{longtable}{p{5em}p{3em}p{.5em}p{7cm}p{6cm}}
\caption{Recognized points of type bigATR for p = 13} \label{tbl:table-p13} \\
\toprule
label & D & n & field & factor \\
\midrule
\endfirsthead
\caption[]{Recognized points of type bigATR for p = 13} \\
\toprule
label & D & n & field & factor \\
\midrule
\endhead
\midrule
\multicolumn{5}{r}{Continued on next page} \\
\midrule
\endfoot
\bottomrule
\endlastfoot
6\_24\_odd & $29$ & $3$ & $x^{8} - 2 x^{7} + 2 x^{6} - 30 x^{5} + 254 x^{4} - 870 x^{3} + 1682 x^{2} - 1682 x + 841$ & $J = 7^{-16} \cdot 719^{4} \cdot 3539^{-4} $ (l) \\
6\_24\_odd & $29$ & $1$ & $x^{8} - 4 x^{7} + 8 x^{6} - 6 x^{5} + 15 x^{4} - 26 x^{3} + 2 x^{2} - 14 x + 49$ & $J = 7^{16} \cdot 83^{4} $ (l) \\
6\_24\_even & $29$ & $1$ & $x^{8} - 4 x^{7} + 8 x^{6} - 6 x^{5} + 15 x^{4} - 26 x^{3} + 2 x^{2} - 14 x + 49$ & $J = 71^{4} $ (l) \\
6\_24\_even & $29$ & $3$ & $x^{8} - 2 x^{7} + 2 x^{6} - 30 x^{5} + 254 x^{4} - 870 x^{3} + 1682 x^{2} - 1682 x + 841$ & $J = 7^{-8} \cdot 59^{4} \cdot 419^{4} $ (l) \\
6\_24\_odd & $181$ & $1$ & $x^{8} - 4 x^{7} + 8 x^{6} + 10 x^{5} + 99 x^{4} - 226 x^{3} + 162 x^{2} - 450 x + 625$ & $J = 3^{8} \cdot 59^{4} \cdot 479^{-4} \cdot 919^{-8} \cdot 1447^{-8} \cdot 1567^{-8} \cdot 2699^{-4} $ (l) \\
6\_24\_even & $181$ & $1$ & $x^{8} - 4 x^{7} + 8 x^{6} + 10 x^{5} + 99 x^{4} - 226 x^{3} + 162 x^{2} - 450 x + 625$ & $J = 3^{-8} \cdot 11^{-20} \cdot 43^{8} \cdot 359^{-4} \cdot 5387^{4} \cdot 6659^{-4} $ (l) \\
6\_24\_odd & $313$ & $1$ & $x^{8} - 4 x^{7} + 8 x^{6} + 14 x^{5} + 177 x^{4} - 390 x^{3} + 242 x^{2} - 1056 x + 2304$ & $J = 2^{56} \cdot 19^{-8} \cdot 599^{4} $ (l) \\
6\_24\_even & $313$ & $1$ & $x^{8} - 4 x^{7} + 8 x^{6} + 14 x^{5} + 177 x^{4} - 390 x^{3} + 242 x^{2} - 1056 x + 2304$ & $J = 2^{-32} \cdot 3^{8} \cdot 11^{-8} \cdot 311^{-4} \cdot 479^{4} $ (l) \\
6\_24\_odd & $337$ & $1$ & $x^{8} - 4 x^{7} + 8 x^{6} + 22 x^{5} + 225 x^{4} - 502 x^{3} + 450 x^{2} - 840 x + 784$ & $J = 2^{-72} \cdot 839^{-4} $ (l) \\
6\_24\_even & $337$ & $1$ & $x^{8} - 4 x^{7} + 8 x^{6} + 22 x^{5} + 225 x^{4} - 502 x^{3} + 450 x^{2} - 840 x + 784$ & $J = 2^{32} \cdot 3^{8} \cdot 7^{8} \cdot 239^{-4} $ (l) \\
6\_24\_odd & $373$ & $1$ & $x^{8} - 4 x^{7} + 8 x^{6} + 26 x^{5} + 267 x^{4} - 594 x^{3} + 578 x^{2} - 714 x + 441$ & $J = 3^{16} \cdot 7^{-16} \cdot 179^{-4} \cdot 283^{8} \cdot 863^{-4} \cdot 1031^{-4} \cdot 3659^{4} \cdot 5939^{4} \cdot 8699^{4} \cdot 9371^{4} $ (l) \\
6\_24\_even & $373$ & $1$ & $x^{8} - 4 x^{7} + 8 x^{6} + 26 x^{5} + 267 x^{4} - 594 x^{3} + 578 x^{2} - 714 x + 441$ & $J = 3^{-16} \cdot 7^{-16} \cdot 107^{-4} \cdot 383^{8} \cdot 503^{4} \cdot 1019^{4} \cdot 9923^{4} \cdot 15083^{-4} \cdot 15971^{-4} $ (l) \\
6\_24\_odd & $389$ & $1$ & $x^{8} - 4 x^{7} + 8 x^{6} + 10 x^{5} + 203 x^{4} - 434 x^{3} + 162 x^{2} - 1386 x + 5929$ & $J = 383^{4} \cdot 443^{-4} \cdot 2003^{4} \cdot 2411^{4} $ (l) \\
6\_24\_even & $389$ & $1$ & $x^{8} - 4 x^{7} + 8 x^{6} + 10 x^{5} + 203 x^{4} - 434 x^{3} + 162 x^{2} - 1386 x + 5929$ & $J = 7^{-8} \cdot 11^{-4} \cdot 59^{-4} \cdot 347^{-4} \cdot 983^{-4} \cdot 7211^{-4} $ (l) \\
6\_24\_odd & $521$ & $1$ & $x^{8} - 4 x^{7} + 8 x^{6} + 30 x^{5} + 369 x^{4} - 806 x^{3} + 722 x^{2} - 1520 x + 1600$ & $J = 2^{96} \cdot 211^{8} \cdot 283^{-8} $ (l) \\
6\_24\_even & $521$ & $1$ & $x^{8} - 4 x^{7} + 8 x^{6} + 30 x^{5} + 369 x^{4} - 806 x^{3} + 722 x^{2} - 1520 x + 1600$ & $J = 2^{8} \cdot 11^{8} \cdot 47^{12} \cdot 719^{4} $ (l) \\
6\_24\_odd & $701$ & $1$ & $x^{8} - 4 x^{7} + 8 x^{6} + 42 x^{5} + 567 x^{4} - 1226 x^{3} + 1250 x^{2} - 950 x + 361$ & $J = 7^{-8} \cdot 31^{8} \cdot 43^{-8} \cdot 131^{4} \cdot 29819^{-4} $ (l) \\
6\_24\_even & $701$ & $1$ & $x^{8} - 4 x^{7} + 8 x^{6} + 42 x^{5} + 567 x^{4} - 1226 x^{3} + 1250 x^{2} - 950 x + 361$ & $J = 83^{8} \cdot 103^{8} \cdot 1847^{4} \cdot 10499^{4} \cdot 17939^{4} $ (l) \\
6\_24\_odd & $809$ & $1$ & $x^{8} - 4 x^{7} + 8 x^{6} + 46 x^{5} + 665 x^{4} - 1430 x^{3} + 1458 x^{2} - 1080 x + 400$ & $J = 2^{-8} \cdot 7^{-8} \cdot 19^{-8} \cdot 43^{-8} $ (l) \\
6\_24\_even & $809$ & $1$ & $x^{8} - 4 x^{7} + 8 x^{6} + 46 x^{5} + 665 x^{4} - 1430 x^{3} + 1458 x^{2} - 1080 x + 400$ & $J = 2^{32} \cdot 7^{24} \cdot 379^{-8} $ (l) \\
6\_24\_odd & $985$ & $1$ & $x^{8} - 4 x^{7} + 8 x^{6} + 14 x^{5} + 513 x^{4} - 1062 x^{3} + 242 x^{2} - 4752 x + 46656$ & $J^3 = 2^{-32} \cdot 3^{4} \cdot 19^{4} \cdot 67^{4} \cdot 1847^{2} $ (l) \\
6\_24\_even & $985$ & $1$ & $x^{8} - 4 x^{7} + 8 x^{6} + 14 x^{5} + 513 x^{4} - 1062 x^{3} + 242 x^{2} - 4752 x + 46656$ & $J^3 = 2^{-12} \cdot 3^{-4} \cdot 19^{4} \cdot 1367^{-2} \cdot 1487^{-2} \cdot 2663^{-2} \cdot 2687^{2} $ (l) \\
6\_24\_odd & $1153$ & $1$ & $x^{8} - 4 x^{7} + 8 x^{6} + 6 x^{5} + 577 x^{4} - 1174 x^{3} + 98 x^{2} - 3864 x + 76176$ & $J = 2^{8} \cdot 11^{8} \cdot 23^{4} \cdot 503^{-4} \cdot 2663^{-4} $ (l) \\
6\_24\_even & $1153$ & $1$ & $x^{8} - 4 x^{7} + 8 x^{6} + 6 x^{5} + 577 x^{4} - 1174 x^{3} + 98 x^{2} - 3864 x + 76176$ & $J = 2^{-32} \cdot 3^{8} \cdot 11^{8} \cdot 23^{-4} \cdot 1439^{4} \cdot 3023^{4} $ (l) \\
6\_24\_odd & $1249$ & $1$ & $x^{8} - 4 x^{7} + 8 x^{6} + 54 x^{5} + 985 x^{4} - 2086 x^{3} + 1922 x^{2} - 4464 x + 5184$ & $J = 3^{-8} \cdot 103^{8} \cdot 283^{-8} \cdot 547^{8} \cdot 691^{-8} $ (l) \\
6\_24\_even & $1249$ & $1$ & $x^{8} - 4 x^{7} + 8 x^{6} + 54 x^{5} + 985 x^{4} - 2086 x^{3} + 1922 x^{2} - 4464 x + 5184$ & $J = 2^{104} \cdot 3^{-8} \cdot 47^{8} \cdot 79^{-8} \cdot 1103^{-4} $ (l) \\
6\_24\_even & $1453$ & $1$ & $x^{8} - 4 x^{7} + 8 x^{6} + 66 x^{5} + 1267 x^{4} - 2674 x^{3} + 2738 x^{2} - 1554 x + 441$ & $J = 3^{-8} \cdot 7^{-24} \cdot 11^{8} \cdot 23^{-4} \cdot 107^{-8} \cdot 263^{-8} \cdot 1559^{-4} \cdot 2663^{-4} \cdot 3739^{8} \cdot 11003^{-4} \cdot 29123^{-4} \cdot 43499^{-4} \cdot 47459^{4} \cdot 51131^{4} \cdot 56843^{-4} \cdot 59219^{-4} $ (l) \\
6\_24\_odd & $2521$ & $1$ & $x^{8} - 4 x^{7} + 8 x^{6} + 62 x^{5} + 1737 x^{4} - 3606 x^{3} + 2450 x^{2} - 22680 x + 104976$ & $J = 2^{40} \cdot 3^{-16} \cdot 7^{16} \cdot 3359^{-4} $ (l) \\
6\_24\_even & $2521$ & $1$ & $x^{8} - 4 x^{7} + 8 x^{6} + 62 x^{5} + 1737 x^{4} - 3606 x^{3} + 2450 x^{2} - 22680 x + 104976$ & $J = 2^{-64} \cdot 3^{8} \cdot 7^{8} \cdot 71^{8} \cdot 167^{-4} $ (l) \\
\end{longtable}

}
{\tiny
\begin{longtable}{p{5em}p{3em}p{.5em}p{7cm}p{6cm}}
\caption{Recognized points of type bigATR for p = 17} \label{tbl:table-p17} \\
\toprule
label & D & n & field & factor \\
\midrule
\endfirsthead
\caption[]{Recognized points of type bigATR for p = 17} \\
\toprule
label & D & n & field & factor \\
\midrule
\endhead
\midrule
\multicolumn{5}{r}{Continued on next page} \\
\midrule
\endfoot
\bottomrule
\endlastfoot
6\_24\_even & $8$ & $1$ & $x^{8} + 6 x^{4} + 1$ & $J = 31^{4} $ (l) \\
6\_24\_odd & $8$ & $2$ & $x^{8} + 6 x^{4} + 1$ & $J = 31^{-4} $ (l) \\
6\_24\_even & $53$ & $3$ & $x^{8} - 4 x^{7} + 8 x^{6} + 2 x^{5} + 235 x^{4} - 482 x^{3} + 50 x^{2} - 1130 x + 12769$ & $J = 7^{-4} \cdot 11^{8} \cdot 199^{-4} \cdot 523^{4} \cdot 1279^{4} \cdot 10399^{4} $ (l) \\
6\_24\_even & $89$ & $2$ & $x^{8} + 57 x^{4} + 256$ & $J = 9343^{-2} $ (l) \\
6\_24\_odd & $89$ & $2$ & $x^{8} + 57 x^{4} + 256$ & $J = 79^{4} \cdot 131^{4} \cdot 227^{-4} \cdot 1879^{4} \cdot 3343^{-2} $ (l) \\
6\_24\_even & $101$ & $1$ & $x^{8} - 4 x^{7} + 8 x^{6} + 10 x^{5} + 59 x^{4} - 146 x^{3} + 162 x^{2} - 90 x + 25$ & $J = 23^{16} \cdot 211^{4} \cdot 587^{-8} $ (l) \\
6\_24\_odd & $101$ & $1$ & $x^{8} - 4 x^{7} + 8 x^{6} + 10 x^{5} + 59 x^{4} - 146 x^{3} + 162 x^{2} - 90 x + 25$ & $J = 23^{8} \cdot 47^{-16} \cdot 683^{-8} \cdot 1231^{4} $ (l) \\
6\_24\_even & $137$ & $1$ & $x^{8} - 4 x^{7} + 8 x^{6} - 2 x^{5} + 65 x^{4} - 134 x^{3} + 18 x^{2} - 192 x + 1024$ & $J = 2^{-24} $ (l) \\
6\_24\_even & $145$ & $3$ & $x^{8} + 35 x^{6} + 523 x^{4} + 5805 x^{2} + 2916$ & $J^4 = 2^{-4} \cdot 1723 $ (l) \\
6\_24\_odd & $145$ & $3$ & $x^{8} + 35 x^{6} + 523 x^{4} + 5805 x^{2} + 2916$ & $J^4 = 2^{12} \cdot 59^{2} $ (l) \\
6\_24\_even & $229$ & $1$ & $x^{8} - 4 x^{7} + 8 x^{6} - 6 x^{5} + 115 x^{4} - 226 x^{3} + 2 x^{2} - 114 x + 3249$ & $J^3 = 19^{-12} \cdot 307^{4} \cdot 443^{8} \cdot 1427^{-8} \cdot 2287^{4} \cdot 6079^{-4} \cdot 7159^{-4} $ (l) \\
6\_24\_odd & $229$ & $1$ & $x^{8} - 4 x^{7} + 8 x^{6} - 6 x^{5} + 115 x^{4} - 226 x^{3} + 2 x^{2} - 114 x + 3249$ & $J^3 = 3^{8} \cdot 11^{-16} \cdot 19^{-4} \cdot 71^{-8} \cdot 919^{4} \cdot 1499^{-8} $ (l) \\
6\_24\_even & $257$ & $1$ & $x^{8} - 4 x^{7} + 8 x^{6} + 22 x^{5} + 185 x^{4} - 422 x^{3} + 450 x^{2} - 240 x + 64$ & $J^3 = 1 $ (l) \\
6\_24\_odd & $257$ & $1$ & $x^{8} - 4 x^{7} + 8 x^{6} + 22 x^{5} + 185 x^{4} - 422 x^{3} + 450 x^{2} - 240 x + 64$ & $J^3 = 1 $ (l) \\
6\_24\_even & $281$ & $3$ & $x^{8} + 63 x^{6} + 1823 x^{4} + 33861 x^{2} + 96100$ & $J = 2^{-120} \cdot 167^{8} \cdot 211^{-12} $ (l) \\
6\_24\_odd & $281$ & $3$ & $x^{8} + 63 x^{6} + 1823 x^{4} + 33861 x^{2} + 96100$ & $J = 2^{24} \cdot 4219^{4} $ (l) \\
6\_24\_even & $373$ & $1$ & $x^{8} - 4 x^{7} + 8 x^{6} + 26 x^{5} + 267 x^{4} - 594 x^{3} + 578 x^{2} - 714 x + 441$ & $J = 7^{4} \cdot 59^{-8} \cdot 103^{-8} \cdot 419^{-8} \cdot 571^{-4} \cdot 787^{-4} \cdot 3919^{-4} \cdot 10711^{-4} $ (l) \\
6\_24\_odd & $373$ & $1$ & $x^{8} - 4 x^{7} + 8 x^{6} + 26 x^{5} + 267 x^{4} - 594 x^{3} + 578 x^{2} - 714 x + 441$ & $J = 3^{8} \cdot 7^{-8} \cdot 31^{-8} \cdot 691^{-4} \cdot 2311^{-4} \cdot 6271^{-4} $ (l) \\
6\_24\_even & $389$ & $1$ & $x^{8} - 4 x^{7} + 8 x^{6} + 10 x^{5} + 203 x^{4} - 434 x^{3} + 162 x^{2} - 1386 x + 5929$ & $J = 7^{-12} \cdot 19^{4} $ (l) \\
6\_24\_odd & $389$ & $1$ & $x^{8} - 4 x^{7} + 8 x^{6} + 10 x^{5} + 203 x^{4} - 434 x^{3} + 162 x^{2} - 1386 x + 5929$ & $J = 7^{-8} \cdot 127^{8} \cdot 787^{-4} \cdot 6967^{-4} $ (l) \\
6\_24\_even & $421$ & $1$ & $x^{8} - 4 x^{7} + 8 x^{6} + 18 x^{5} + 247 x^{4} - 538 x^{3} + 338 x^{2} - 1638 x + 3969$ & $J = 7^{8} \cdot 103^{-4} \cdot 643^{4} \cdot 11071^{-4} $ (l) \\
6\_24\_odd & $421$ & $1$ & $x^{8} - 4 x^{7} + 8 x^{6} + 18 x^{5} + 247 x^{4} - 538 x^{3} + 338 x^{2} - 1638 x + 3969$ & $J = 3^{4} \cdot 79^{-8} \cdot 199^{4} $ (l) \\
6\_24\_even & $457$ & $1$ & $x^{8} - 4 x^{7} + 8 x^{6} - 2 x^{5} + 225 x^{4} - 454 x^{3} + 18 x^{2} - 672 x + 12544$ & $J = 2^{-40} \cdot 3^{-4} \cdot 19^{-4} $ (l) \\
6\_24\_odd & $457$ & $1$ & $x^{8} - 4 x^{7} + 8 x^{6} - 2 x^{5} + 225 x^{4} - 454 x^{3} + 18 x^{2} - 672 x + 12544$ & $J = 2^{8} \cdot 7^{-8} \cdot 19^{4} $ (l) \\
6\_24\_even & $557$ & $1$ & $x^{8} - 4 x^{7} + 8 x^{6} + 18 x^{5} + 315 x^{4} - 674 x^{3} + 338 x^{2} - 2522 x + 9409$ & $J = 43^{4} \cdot 139^{4} \cdot 467^{8} \cdot 3271^{-4} \cdot 13183^{4} \cdot 15679^{-4} $ (l) \\
6\_24\_odd & $557$ & $1$ & $x^{8} - 4 x^{7} + 8 x^{6} + 18 x^{5} + 315 x^{4} - 674 x^{3} + 338 x^{2} - 2522 x + 9409$ & $J = 67^{8} \cdot 8839^{-4} $ (l) \\
6\_24\_even & $577$ & $1$ & $x^{8} - 4 x^{7} + 8 x^{6} + 38 x^{5} + 465 x^{4} - 1014 x^{3} + 1058 x^{2} - 552 x + 144$ & $J^7 = 3^{-8} \cdot 19^{-8} \cdot 523^{4} \cdot 1123^{4} $ (l) \\
6\_24\_odd & $577$ & $1$ & $x^{8} - 4 x^{7} + 8 x^{6} + 38 x^{5} + 465 x^{4} - 1014 x^{3} + 1058 x^{2} - 552 x + 144$ & $J^7 = 3^{-4} \cdot 19^{4} $ (l) \\
6\_24\_even & $1069$ & $1$ & $x^{8} - 4 x^{7} + 8 x^{6} + 50 x^{5} + 843 x^{4} - 1794 x^{3} + 1682 x^{2} - 3306 x + 3249$ & $J = 3^{-4} \cdot 67^{-8} \cdot 71^{-8} \cdot 79^{-8} \cdot 331^{8} \cdot 12799^{-4} $ (l) \\
6\_24\_odd & $1069$ & $1$ & $x^{8} - 4 x^{7} + 8 x^{6} + 50 x^{5} + 843 x^{4} - 1794 x^{3} + 1682 x^{2} - 3306 x + 3249$ & $J = 3^{-8} \cdot 59^{-8} \cdot 79^{4} \cdot 379^{-4} \cdot 739^{4} \cdot 17839^{4} \cdot 19759^{4} \cdot 22039^{-4} $ (l) \\
6\_24\_even & $1249$ & $1$ & $x^{8} - 4 x^{7} + 8 x^{6} + 54 x^{5} + 985 x^{4} - 2086 x^{3} + 1922 x^{2} - 4464 x + 5184$ & $J = 2^{-72} \cdot 3^{12} \cdot 239^{-8} \cdot 503^{8} \cdot 1867^{-4} $ (l) \\
6\_24\_odd & $1249$ & $1$ & $x^{8} - 4 x^{7} + 8 x^{6} + 54 x^{5} + 985 x^{4} - 2086 x^{3} + 1922 x^{2} - 4464 x + 5184$ & $J = 2^{-56} \cdot 3^{-8} \cdot 31^{8} \cdot 499^{-4} $ (l) \\
6\_24\_odd & $1301$ & $1$ & $x^{8} - 4 x^{7} + 8 x^{6} + 42 x^{5} + 867 x^{4} - 1826 x^{3} + 1250 x^{2} - 8450 x + 28561$ & $J = 23^{8} \cdot 43^{4} \cdot 271^{-4} \cdot 607^{-4} \cdot 2131^{-4} \cdot 2683^{-4} \cdot 3331^{8} \cdot 25303^{-4} $ (l) \\
6\_24\_odd & $1549$ & $1$ & $x^{8} - 4 x^{7} + 8 x^{6} + 26 x^{5} + 855 x^{4} - 1770 x^{3} + 578 x^{2} - 10710 x + 99225$ & $J = 59^{16} \cdot 71^{-8} \cdot 179^{-8} \cdot 419^{-16} \cdot 1867^{-4} \cdot 3259^{-4} \cdot 3467^{-8} \cdot 8291^{8} \cdot 15679^{4} \cdot 40519^{4} $ (l) \\
6\_24\_even & $1585$ & $1$ & $x^{8} - 4 x^{7} + 8 x^{6} + 6 x^{5} + 793 x^{4} - 1606 x^{3} + 98 x^{2} - 5376 x + 147456$ & $J = 2^{32} \cdot 3^{-2} \cdot 11^{-4} \cdot 79^{4} \cdot 787^{2} $ (l) \\
6\_24\_odd & $1585$ & $1$ & $x^{8} - 4 x^{7} + 8 x^{6} + 6 x^{5} + 793 x^{4} - 1606 x^{3} + 98 x^{2} - 5376 x + 147456$ & $J = 3^{-2} \cdot 11^{-4} \cdot 59^{-4} \cdot 163^{2} \cdot 2683^{-2} \cdot 2803^{2} $ (l) \\
6\_24\_even & $1657$ & $1$ & $x^{8} - 4 x^{7} + 8 x^{6} + 62 x^{5} + 1305 x^{4} - 2742 x^{3} + 2450 x^{2} - 7560 x + 11664$ & $J = 2^{-8} \cdot 47^{-8} \cdot 107^{8} \cdot 739^{-4} \cdot 2203^{4} \cdot 3331^{4} $ (l) \\
6\_24\_odd & $1657$ & $1$ & $x^{8} - 4 x^{7} + 8 x^{6} + 62 x^{5} + 1305 x^{4} - 2742 x^{3} + 2450 x^{2} - 7560 x + 11664$ & $J = 2^{-8} \cdot 3^{8} \cdot 23^{8} \cdot 59^{8} \cdot 71^{-8} $ (l) \\
6\_24\_even & $1789$ & $1$ & $x^{8} - 4 x^{7} + 8 x^{6} + 74 x^{5} + 1575 x^{4} - 3306 x^{3} + 3362 x^{2} - 2214 x + 729$ & $J = 3^{-8} \cdot 7^{-24} \cdot 23^{-8} \cdot 127^{16} \cdot 1231^{4} \cdot 1327^{-12} \cdot 1459^{4} \cdot 2383^{12} \cdot 6967^{4} \cdot 7639^{4} \cdot 9491^{-8} \cdot 10739^{8} \cdot 27967^{-4} \cdot 43399^{-4} \cdot 47119^{-4} \cdot 56167^{4} \cdot 56359^{4} $ (l) \\
6\_24\_odd & $1889$ & $1$ & $x^{8} - 4 x^{7} + 8 x^{6} + 70 x^{5} + 1553 x^{4} - 3254 x^{3} + 3042 x^{2} - 7176 x + 8464$ & $J = 2^{88} \cdot 43^{-8} $ (l) \\
6\_24\_even & $1985$ & $1$ & $x^{8} - 4 x^{7} + 8 x^{6} + 54 x^{5} + 1353 x^{4} - 2822 x^{3} + 1922 x^{2} - 15872 x + 65536$ & $J = 2^{16} \cdot 11^{4} \cdot 131^{-4} $ (l) \\
6\_24\_odd & $1985$ & $1$ & $x^{8} - 4 x^{7} + 8 x^{6} + 54 x^{5} + 1353 x^{4} - 2822 x^{3} + 1922 x^{2} - 15872 x + 65536$ & $J = 2^{-48} \cdot 3067^{2} \cdot 3163^{2} $ (l) \\
6\_24\_even & $2593$ & $1$ & $x^{8} - 4 x^{7} + 8 x^{6} + 86 x^{5} + 2217 x^{4} - 4614 x^{3} + 4418 x^{2} - 9024 x + 9216$ & $J = 2^{40} \cdot 19^{4} \cdot 79^{8} \cdot 1171^{4} \cdot 3019^{-4} \cdot 3067^{4} \cdot 5227^{4} $ (l) \\
6\_24\_odd & $2593$ & $1$ & $x^{8} - 4 x^{7} + 8 x^{6} + 86 x^{5} + 2217 x^{4} - 4614 x^{3} + 4418 x^{2} - 9024 x + 9216$ & $J = 2^{-136} \cdot 3^{-4} \cdot 19^{4} \cdot 31^{-8} \cdot 67^{8} \cdot 103^{-8} \cdot 167^{8} \cdot 2131^{-4} $ (l) \\
\end{longtable}

}
\end{landscape}

\bibliographystyle{amsalpha}
\bibliography{refs}

\end{document}